\documentclass[a4paper,11pt]{amsart}
\usepackage{amsmath,amssymb,amscd,amsfonts}
\usepackage{amssymb}
\usepackage{amsrefs, esint}
\usepackage[T1]{fontenc}
\usepackage{tikz-cd}
\numberwithin{equation}{section}
\newtheorem{theorem}{Theorem}[section]

\newtheorem{rem}[theorem]{Remark}
\newtheorem{defn}[theorem]{Definition}
\newtheorem{lemma}[theorem]{Lemma}

\newtheorem{prop}[theorem]{Proposition}

\newtheorem{cor}[theorem]{Corollary}
\newtheorem{example}[theorem]{Example}

\def\Im{\mathop{\rm Im}\nolimits}

\def\dbar{\bar\partial}
\def\ddbar{\partial\bar\partial}
\def\d{\partial}

\def\cI{{\mathcal I}}

\def\cE{{\mathcal E}}

\def\cM{{\mathcal M}}

\def\cD{{\mathcal D}}

\def\cF{{\mathcal F}}

\let\ep=\varepsilon
\let\vp=\varphi 

\def\bC{{\mathbb C}}
\def\bR{{\mathbb R}}

\def\bP{{\mathbb P}}
\def\a{{\alpha}}
\def\z{{\zeta}}

\def\b{\beta}

\def\k{{\kappa}}

\def\l{{\ell}}

\def\ba{{\mathbf{a} }}
\def\bb{{\mathbf{b} }}
\def\bc{{\mathbf{c} }}
\def\be{{\mathbf{e} }}

\title[Zero mass conjecture]
{On the residual Monge-Amp\`{e}re mass of plurisubharmonic functions with symmetry, II}

\author{Weiyong He}
\address{Department of Mathematics, University of Oregon, Eugene, OR, USA, 97403.}
\email{whe@uoregon.edu}

\author{Long Li}
\address{ Institute of Mathematical Science at ShanghaiTech University, 393 Middle Huaxia Road, Pudong, Shanghai, China, 201210.}
\email{lilong1@shanghaitech.edu.cn}

\author{Xiaowei Xu}
\address{School of Mathematical Sciences, USTC, Hefei, Anhui, China, 230026; CAS, Wu Wen-Tsun Key Laboratory of Mathematics.}
\email{xwxu09@ustc.edu.cn}

\begin{document}
\maketitle 

$$ \emph{In memory of Prof. Demailly} $$

\begin{abstract}
The aim of this article is to study the residual Monge-Amp\`{e}re mass 
of a plurisubharmonic function with an isolated singularity, provided with the circular symmetry. 
With the aid of Sasakian geometry, 
we obtain an estimate on the residual mass of this function 
with respect to its Lelong number and maximal directional Lelong number. 
This result partially answers the zero mass conjecture raised by Guedj and Rashkovskii. 
\end{abstract}

\section{Introduction}
\bigskip

The zero mass conjecture for plurisubharmonic functions, 
raised by Guedj and Rashkovskii, 
is a fundamental but difficult problem in pluripotential theory, see \cite{GZ15}, \cite{Ra16}.
It states that the Lelong number of a plurisubharmonic function (with an isolated singularity)
must be positive, if its complex Monge-Amp\`ere measure has a Dirac mass at the singularity. 
Equivalently, this says that the residual Monge-Amp\`ere mass at the singularity of the plurisubharmonic function 
must be zero, if it has zero Lelong number at this point.

In literature, there have been many important contributions towards this conjecture, 
including works like \cite{Ceg02}, \cite{Ra01}, \cite{Ra06}, \cite{Ra13}, \cite{KR21}, \cite{BFJ07}, \cite{ACH19} and \cite{G10}. 
In particular, Rashkovskii \cite{Ra01} confirms it, 
provide that the plurisubharmonic function has the \emph{toric symmetry}. 
As a generalization of the toric symmetry, 
the so called \emph{circular symmetry}   
has been recently studied, 
and the zero mass conjecture is confirmed for 
all circular symmetric plurisubharmonic functions in $\bC^2$, see \cite{BB},  \cite{Li19} and \cite{Li23}.

In this paper, we study the zero mass conjecture for plurisubharmonic functions with circular symmetry in higher dimensions. 
Consider the family $\cF(B_1)$
consisting of all  circular symmetric (or $S^1$-invariant)
plurisubharmonic functions on the unit ball $B_1$ in $\bC^{n+1}$
that are locally bounded outside the origin, see Definition \ref{def-pfc-001}. 
Moveover, 
we denote $\cF^{\infty}(B_1)$ by the sub-collection of $\cF(B_1)$
that requires $C^2$-continuity of the functions outside the origin, see Definition \ref{def-pfc-002}.
In the case of complex dimension two, the second-named author proved  the following estimate. 

\begin{theorem}[\cite{Li23}]
\label{thm-int-001}
For a function $u\in \cF(B_1)$ in $\bC^2$, we have 
\begin{equation}
\label{int-001}
[\nu_u(0)]^2 \leq \tau_{u}(0) \leq 2 \lambda_u(0) \cdot \nu_{u}(0) + [\nu_u(0)]^2.
\end{equation}
\end{theorem}

Here $\nu_u(0)$ and $\tau_u(0)$ denote 
the Lelong number and the residual Monge-Amp\`ere mass 
of $u$ at the origin, respectively. 
Moreover, the constant $\lambda_u(0)$ is 
the \emph{maximal directional Lelong number} of $u$ at the origin. 
Let $\l_\z: = \bC\cdot \z$ be the complex line through the origin
in the complex direction $\z\in (\bC^{n+1})^*$.
This terminology stems from taking the supreme  of 
the Lelong number at the origin  
 of the restriction $u|_{\l_\z}$
 among all complex directions.
For more details, see Definition  \ref{def-pfc-003} and \ref{def-pfc-004}.


However, taking supremes does not guarantee the finiteness in general. 
Fortunately, it turns out that 
$\lambda_u(0)$ 
is indeed a non-negative real number for any function $u\in \cF(B_1)$, see Proposition \ref{prop-pfc-001}. 
Then the zero mass conjecture  in $\bC^2$  directly follows from the estimate in equation (\ref{int-001}).

More precisely, this estimate is due to  a  decomposition formula 
of the complex Monge-Amp\`ere measure, see Theorem 4.4, \cite{Li23}.
To obtain this formula, 
we have utilized a particular local coordinate system (called the complex Hopf-coordinate) on $\bC^2$,
and it is naturally induced from the famous Hopf-fiberation:
\\

\begin{tikzcd}
\ \ \ \ \ \ \ \ \ \ \ \ \ \ \ \ \ \ \ \ \ \ \ \ \ \ \ \ \ S^1 \arrow[r, hook] & S^{3} \arrow[r, "p"] & \bC\bP^1. 
\end{tikzcd}
\\

In higher dimesions, the new observation is that there is a natural  
\emph{K\"ahler cone structure} on $(\bC^{n+1})^* = S^{2n+1} \times \bR_+$.
That is to say, 
we can consider the \emph{ standard Sasakian structure} of the unit sphere $S^{2n+1}$
with the base manifold $\bC\bP^n$. 
It is nothing but the Hopf-fiberation 
as a principle circle bundle over the complex projective space: 
\\

\begin{tikzcd}
\ \ \ \ \ \ \ \ \ \ \ \ \ \ \ \ \ \ \ \ \ \ \ \ \ \ \ \ \ S^1 \arrow[r, hook] & S^{2n+1} \arrow[r, "p"] & \bC\bP^n. 
\end{tikzcd}
\\

Sasakian geometry is served as a bridge to connect with
the two K\"ahler structures respectively of the cone and the base manifold. 
It enables us to achieve the decomposition formula (Theorem \ref{thm-df-001}) in any dimension.
Comparing to the domain case,
this formula
can be viewed as the push forward of the complex Monge-Amp\`ere measure of a $u\in \cF^{\infty}(B_1)$
to the base manifold, under the Sasakian structure of $S^{2n+1}$, see Section \ref{sub-007}. 
For more discussion on Sasakian manifolds, the readers are referred to 
\cite{SH62}, \cite{MSY08}, \cite{BG08}, \cite{HS16} and \cite{HL21}.

Finally, 
we come up with the following estimate on the residual Monge-Amp\`ere mass.

\begin{theorem}[Theorem \ref{thm-gr-002}]
\label{thm-int-002}
For a function $u\in \cF(B_1)$, 
there exists a dimensional constant $C_n \geq n+1$ such that we have 
\begin{equation}
\label{int-002}
[\nu_u(0)]^{n+1} \leq \tau_{u}(0) \leq 2 C_n [\lambda_u(0)]^n \cdot \nu_u(0).
\end{equation}
\end{theorem}
The lefthand side of the inequality in equation \eqref{int-002}
was indicated by Cegrell \cite{Ceg04},
and the zero mass conjecture follows from the righthand side.

\begin{theorem}[Theorem \ref{cor-gr-003}]
\label{cor-int-001}
For a function $u\in \cF(B_1)$, we have 
\begin{equation}
\label{int-003}
\nu_u(0) =0 \Rightarrow \tau_u(0) =0. 
\end{equation}
\end{theorem}

The estimate in equation \eqref{int-002} is  stronger than the case in complex dimension two. 
This is because of a better understanding of the positivity conditions, see Lemma \ref{lem-ps-001}. 
However, this inequality is not sharp.
Moreover, 
this estimate fails if the plurisubharmonic function is no longer circular symmetric, see Example 6.13, \cite{Li23}.

In the last section, 
we provide a different proof of the decomposition formula
via Cartan's method of moving frames, see \cite{C87}, \cite{Chern}. 
This should be useful in the future work when there is no symmetry conditions. 

\bigskip

\textbf{Acknowledgment: }
We are very grateful to Prof. Xiuxiong Chen 
for his continuous support and encouragement in mathematics. 
This problem was raised to the second-named author when 
he was studying with Prof. Demailly in Fourier Institute, Grenoble. 
It is also a great pleasure to discuss with Song Sun, Chengjian Yao and Jian Wang.

The third-named author is supported by the NSFC (No. 11871445), 
the Stable Support for Youth Team in Basic Research Field, 
CAS(YSBR-001) and the Fundamental Research Funds for the Central Universities.

\bigskip
\smallskip

\section{Plurisubharmonic functions with circular symmetry}
\smallskip

In this section, we recall a few basic facts about 
$S^1$-invariant plurisubharmonic functions.
Denote by $z: = (z^0, \cdots, z^n)$  the complex Euclidean coordinate on $\bC^{n+1}$.
There is a natural $S^1$-action that sends 
$$ z \rightarrow e^{i\theta} z: = (e^{i\theta}z^0, \cdots, e^{i\theta} z^n), $$
for all angles $\theta\in \bR$. 
A domain $D$ is called balanced if it is invariant under this $S^1$-action, 
and a function $u$ on a balanced domain is said to be circular symmetric 
or $S^1$-invariant if it satisfies 
$ u(z) = u(e^{i\theta}z). $

Assume that $D$ contains the origin. 
We say that a plurisubharmonic function $u$ on $D$
has an isolated singularity at the origin,
if it is locally bounded on the punctured domain $D^*: = D - \{ 0\}$ and $u(0) = -\infty$.
Then we can consider the following two families of plurisubharmonic functions.

\begin{defn}
\label{def-pfc-001}
An $S^1$-invariant plurisubharmoinc function belongs to the family $\cF(D)$,
if it is $L^{\infty}_{loc}$ on $D^*$. 
\end{defn}

\begin{defn}
\label{def-pfc-002}
An $S^1$-invariant plurisubharmoinc function belongs to the family $\cF^{\infty}(D)$,
if it is $C^2$-continuous on $D^*$. 
\end{defn}
We  adopt  the following normalization condition, $ \sup_D u = -1,$
possibly after shrinking $D$ to a smaller balanced domain.

 In the previous work \cite{Li23},
 we have discovered several useful properties of a function $u\in \cF(D)$. 
Although these properties were stated in complex dimension two, 
they are adapted to all dimensions. We recall these facts.

\subsection{The residual mass}
Let $B_R\subset\bC^{n+1}$ be the open ball centered at the origin with radius $R>0$, 
and $B_R^*$ be the corresponding punctured ball. 
Denote $S_R$ by the boundary of the ball, satisfying the equation
\begin{equation}
\label{pfc-001}
\sum_{A=0}^n |z^A|^2 = R^2. 
\end{equation}

In this paper, we always take $D: =B_1$, and focus on the local behaviors 
of a function near the origin. 
Consider a function $u\in \cF(B_1)$,
and then its complex Monge-Amp\`ere measure is a closed positive $(n+1, n+1)$-current:
\begin{equation}
\label{pfc-002}
 \mbox{MA}(u): = dd^c u\wedge\cdots\wedge dd^c u = (dd^c u)^{n+1}, 
 \end{equation}
where the wedge is taken in the sense of Demailly and Bedford-Talyor.
For more details, see \cite{Dem93}, \cite{BT}, \cite{BT0},\cite{Ceg04} and \cite{Ceg98}.
Here we have used the notation 
$$d: = \d + \dbar; \ \ \ \ d^c: = \frac{i}{2}(\dbar -\d). $$
Fixing an $R\in (0,1)$, we take this measure on the ball as 
$$ \mbox{MA}(u)(B_R): = \int_{B_R} (dd^c u)^{n+1}. $$
The residual Monge-Amp\`ere mass of $u$ at the origin is equal to the limit: 
\begin{eqnarray}
\label{pfc-003}
\tau_u(0) &=& 
 \frac{1}{\pi^{n+1}} \mbox{MA}(u)(\{ 0 \})
\nonumber\\
&=& \frac{1}{\pi^{n+1}} \lim_{R\rightarrow 0} \mbox{MA}(u)(B_R). 
\end{eqnarray}

Next we introduce the standard regularization of a plurisubharmonic function. 
Let $\rho(z): = \rho(|z|)$ be a non-negative smooth mollifier in $\bC^{n+1}$
satisfying $\rho(r) = 0$ for all $r\geq 1$, and 
$$ \int_{\bC^{n+1}} \rho(z) \ d\lambda (z) =1, $$
where $d\lambda$ is the Lebesgue measure. 
Take its rescaling for each $\ep >0$ small as 
$$ \rho_{\ep}(z): =  \ep^{-2n-2} \rho(z / \ep). $$

For a function $u\in \cF(B_1)$, 
we can define its regularization 
as a sequence of smooth plurisubharmonic functions that decreases to $u$ pointwise: 
\begin{eqnarray}
\label{pfc-004}
u_{\ep}(z): &=&  (u * \rho_{\ep})(z)
\nonumber\\
&=& \int_{|z-y| \leq \ep} \rho_{\ep} (z-y) u(y) d\lambda(y)
\nonumber\\
&=& \int_{|w|\leq 1} u(z - \ep w) \rho(w) d\lambda(w).
\end{eqnarray}
This sequence converges to $u$ in $C^2$-norm on any compact subset of $B^*_{1}$
if $u\in \cF^{\infty}(B_1)$. Fixing a small $\delta>0$, 
the regularization $u_{\ep}$ is  in $\cF^{\infty}(B_{1-\delta})$ for all $\ep$ small enough. 
Then we have the following convergence of the Monge-Amp\`ere masses. 

\begin{lemma}[\cite{Li23}]
\label{lem-pfc-001}
For a function $u\in \cF(B_1)$, 
the complex Monge-Amp\`ere measure of its regularization $u_{\ep}$ converges on a ball as 
\begin{equation}
\label{pfc-005}
\emph{MA}(u) (B_R) = \lim_{\ep \rightarrow 0^+} \emph{MA}(u_{\ep})(B_R),
\end{equation}
for almost all $R\in (0,1)$. 
\end{lemma}
This convergence follows  the Portemanteau type inequalities below and the fact that the Monge-Amp\`ere measure $\mbox{MA}(u)$ 
vanishes on the boundary hypersphere $S_R$ for almost all $R\in (0,1)$. In fact we have 
\begin{equation}
\label{pfc-006}
\begin{split}
&\mbox{MA}(u)(B_R) \leq \liminf_{\ep\rightarrow 0^+}   \mbox{MA}(u_{\ep})(B_R), \text{on any open ball}\; B_R \subsetneq B_1\\
&\mbox{MA}(u)(\overline{B}_R) \geq \limsup_{\ep\rightarrow 0^+}   \mbox{MA}(u_{\ep})(\overline{B}_R), \text{on any closed ball}\; \overline{B_R}\subset B_1.
\end{split}
\end{equation}

We note that the convergence in Lemma \ref{lem-pfc-001} holds for every $R\in (0,1)$, if we assume $u\in \cF^{\infty}(B_1)$.
Then by the Stokes Theorem we have

\begin{lemma}[\cite{Li23}]
\label{lem-pfc-002}
For $u\in \cF^{\infty}(B_1)$ and for all $R\in (0,1)$, 
\begin{equation}
\label{pfc-008}
\int_{B_R} (dd^c u)^{n+1} = \int_{S_R} d^c u \wedge (dd^c u)^n,
\end{equation}

\end{lemma}

\subsection{Maximal directional Lelong numbers}
Denote $S_u(0, r)$ to be the average of $u$ on the sphere $S_r$,
\begin{equation}
\label{pfc-009}
S_u(0,r):=  \frac{1}{a_{2n+1}}   \int_{|\xi|=1} u(r\xi) d\sigma(\xi), 
\end{equation}
where $d\sigma$ is the area form of the unit sphere $S^{2n+1}$, 
and $a_{2n+1}$ is its total area. 
The Lelong number of a plurisubharmonic function $u$
at the origin is defined to be the following limit 
$$ \nu_u(0) = \lim_{r\rightarrow 0^+} \nu_u(0, r),\; \text{with}\;  \nu_u(0,r):= r\d_r^- S_u(0, r). $$

In fact,  $S_u(0, r)$ is a convex and non-decreasing function of $t: = \log r $, 
and hence the limit $\nu_u(0)$ always exits. 
Assume the function $u$ is in the family $\cF(B_1)$ from now on. 
Denote $\l_\z$ by the complex line through the origin 
in the complex direction $\z\in \bC\bP^n$ as 
$$\l_{\z}: = \bC\cdot [\z], $$
where $[\z]$ means a homogeneous coordinate of $\z$ in $(\bC^{n+1})^*$. 
Thanks to the plurisubharmonicity and $S^1$-symmetry,
the restriction $u|_{\l_\z}$
is a convex and non-decreasing function of  the variable $t\in (-\infty, 0)$. 

It is more convenient to use a parametrization $(r, \theta, \z, \bar\z)$
of the space $(\bR^{2n+2})^* \cong (\bC^{n+1})^*$ 
induced by the fiber maps of the Hopf-fiberation 
\\

\begin{tikzcd}
\ \ \ \ \ \ \ \ \ \ \ \ \ \ \ \ \ \ \ \ \ \ \ \ \ \ \ \ \ S^1 \arrow[r, hook] & S^{2n+1} \arrow[r, "p"] & \bC\bP^n.
\end{tikzcd}
\\


In this parametrization,
the $r$-variable denotes the radius function, and 
$\theta$ stands for the direction induced by the $S^1$-action. 
Then we can re-write the function $u$ under this parametrization as
$$ u_t(\z): = \hat{u}(t, \z, \bar\z) = u(e^t, \z, \bar\z), $$
and denote its derivative with respect to $t$ by  
$$ \dot{u}_t(\z):= \d_t \hat{u}(t, \z,\bar\z) = \frac{d}{dt} u|_{\l_{\z}}. $$
Then it follows for each $\z\in \bC\bP^n$ fixed
\begin{equation}
\label{pfc-010}
\dot{u}_t \geq 0; \ \ \ \ \ddot{u}_t \geq 0,
\end{equation}
for almost all $t\in (-\infty, 0)$.
Moreover, the Lelong number at zero of the restriction $u|_{\l_{\z}}$ 
is equal to 
$$ \nu_{u|_{\l_{\z}}} (0) = \lim_{t\rightarrow -\infty} \dot{u}_t (\z) $$
It is a well-known fact that
the Lelong number of a plurisubharmonic function is invariant 
under the restriction to almost all complex directions in $\bC\bP^{n}$. 
This means that we have 
$$ \nu_u(0) = \nu_{u|_{\l_{\z}}} (0),$$
for almost all $\z\in \bC\bP^n$. 
In particular, the Lelong number $\nu_u(0)$ is the infimum of all such restrictions.
Then the following result follows from the dominated convergence theorem. 

\begin{lemma}[\cite{Li23}]
\label{lem-pfc-003}
For any $u\in \cF^{\infty}(B_1)$, we have 
\begin{equation}
\label{pfc-0101}
[\nu_u(0)]^p = \lim_{t\rightarrow -\infty} \frac{1}{\pi^n} \int_{\bC\bP^n} (\dot{u}_t)^p \omega_{FS}^n, 
\end{equation}
for all $p = 1, \cdots, n+1$.
Here $\omega_{FS}$ stands for the Fubini-Study metric on $\bC\bP^n$ with the normalization 
$$\int_{\bC\bP^n} \omega_{FS}^n = \pi^n.  $$
\end{lemma}

In order to apply the dominated convergence theorem to a general $u\in \cF(B_1)$, 
we actually need an upper bound of $\dot{u}_t(\z)$ for all $\z\in\bC\bP^n$. 
This leads us to consider the supremum of all such restrictions. 

\begin{defn}
\label{def-pfc-003}
The maximal  directional Lelong number of a function $u\in \cF(B_1)$ at a distance $A>0$ is defined to be the following 
\begin{equation}
\label{pfc-011}
M_A(u): = \sup_{\z\in \bC\bP^n} \d_t^+ u_t(\z) |_{t=-A}\in [0, +\infty]. 
\end{equation}
\end{defn}

Thanks to the log-convexity of each restriction $u|_{\l_{\z}}$,
it is apparent that the number
$M_A(u)$ is well-defined, and is non-negative and non-increasing in $A$.
Then we can take its limit as $A\rightarrow +\infty$. 

\begin{defn}
\label{def-pfc-004}
The maximal directional Lelong number of a function $u\in \cF(B_1)$ at the origin is defined to be the following
\begin{equation}
\label{pfc-012}
\lambda_u(0): = \lim_{A\rightarrow +\infty} M_A(u). 
\end{equation}
\end{defn}

A fundamental fact about $\lambda_u(0)$ is that for a $u\in\cF^{\infty}(B_1)$, 
 $\lambda_u(0)$ is always finite, which
is crucial in our estimate of the residual Monge-Amp\`ere mass .
For the convenience of readers, 
we will recall the proof as follows.

\begin{prop}[\cite{Li23}]
\label{prop-pfc-001}
For a function $u\in \cF(B_1)$, 
its maximal directional Lelong number $M_A(u)$ is finite for all $A>0$.
In particular, we have 
\begin{equation}
\label{pfc-013}
0\leq \lambda_u(0) < +\infty.
\end{equation}
\end{prop}
\begin{proof}
For each hypersphere $S_R$ with $R\in (0,1)$, 
there is a constant $C_R>0$ such that we have 
$$ u|_{S_R} > - C_R. $$ 
This is because $u$ is an $L^{\infty}_{loc}$-function defined everywhere in $B^*_1$, 
and hence its restriction to the hypersphere is 
in the space $L^{\infty}(S_R)$.

Suppose on the contrary, we have $M_A(u) = +\infty$
for some $A>0$. Then there exists a subsequence of points $\z_j \in \bC\bP^n$
such that the slope  
\begin{equation}
\label{fit-001}
\d_t^+ u_t(\z_j) |_{t = -A} = \frac{d}{dt}|_{t= (-A)^+} \left(  u|_{\l_{\z_j}}\right)
\end{equation}
diverges to $+\infty$ as $j\rightarrow +\infty$.
In particular, we can pick up a point $\xi$ among this subsequence satisfying 
\begin{equation}
\label{fit-002}
\d_t^+ u_t(\xi) |_{t = -A} > \frac{2C_R}{A},
\end{equation}
for $R = e^{-A}$.
However, as a convex function of $t$, the graph of $u_t(\xi)$ 
is above the following straight line 
$$  y(x) = \frac{2C_R}{A} (x+A) -C_R,$$
for all $t\in [-A, -A/2]$. 
Hence we conclude 
$$ u_{-A/2}(\xi) \geq y(-A/2) \geq 0. $$
This contradicts to the fact that $u$ is a negative function in $B_1$,
and then our result follows.
\end{proof}



The maximal directional Lelong numbers 
provide a way to measure the difference of the infimums of $u$ on the hyperspheres. 

\begin{lemma}
\label{lem-fit-001}
For a function $u\in \cF(B_1)$ and two constants $1< A_1 < A_2$, 
we have the estimate 
\begin{equation}
\label{fit-003}
- \inf_{S_{R_2}} u \leq \int_{A_1}^{A_2} M_T (u) dT - \inf_{S_{R_1}} u,
\end{equation}
where $R_i = e^{-A_i}$ for $i = 1,2$.
\end{lemma}
\begin{proof}
For a complex direction $\z\in \bC\bP^n$, 
the function $u_t(\z)$ is convex, and hence it is Lipschitz continuous. 
Then it follows from the Fundamental Theorem of Calculus 
\begin{eqnarray}
\label{fit-004}
u_{-A_1}(\z) - u_{-A_2}(\z)  &=& \int_{-A_2}^{-A_1} \dot{u}_t (\z) dt
\nonumber\\
&\leq& \int_{A_1}^{A_2} M_T (u) dT.
\end{eqnarray}
Then we have 
\begin{equation}
\label{fit-005}
\inf_{S_{R_1}} u - u_{-A_2}(\z) \leq \int_{A_1}^{A_2} M_T (u) dT.
\end{equation}
Take a sequence of points $\z_k\in \bC\bP^n$ satisfying 
$$ u_{-A_2}(\z_k) \rightarrow \inf_{S_{R_2}} u, $$
as $k \rightarrow +\infty$. Then our result follows.
\end{proof}
In particular, 
we have $M_A(u) >0$ for all $A>0$, 
if $u$ has an isolated singularity at the origin.

\subsection{Regularization and convergence}
As a direct consequence of Proposition \ref{prop-pfc-001},
$\dot{u}_t  = \d_t \hat{u} = r\d_r u$
is  an $L^{\infty}_{loc}$-function in $B_1^*$.
Thanks to the slicing theory, we first conclude that Lemma \ref{lem-pfc-003} holds for all functions in $\cF(B_1)$.
Then it is legal to introduce the following functionals for almost all $t\in (-\infty, 0)$,  
$$ \cI(u_t): = \int_{\bC\bP^n} u_t\omega_{FS}^n; \ \ \ \ \  I(u_t): = \int_{\bC\bP^n} ( \dot{u}_t) \omega_{FS}^n,  u\in \cF(B_1).$$

We recall a few basic facts about these functionals. 

\begin{enumerate}
\item[\textbf{(i)}] 
the functional $\cI(u_t)$ is a convex and non-decreasing function of $t\in (-\infty, -1)$, 
and it is a primitive of the functional $I(u_t)$ along $t$;
\smallskip
\item[\textbf{(ii)}]
the functional $I(u_t)$ is a non-negative and non-decreasing $L^{\infty}$-function of $t\in (-\infty, -1)$;
\smallskip
\item[\textbf{(iii)}]
the functional $I(u_t)$ converges as $t\rightarrow -\infty$
$$ I(u_t) \rightarrow \pi^n \nu_u(0). $$
\end{enumerate}

Next the standard regularization $u_{\ep}$ will be utilized to perform the approximation of these functionals. 
Write the $t$-derivatives of the regularization as 
$$ \dot{u}_{\ep, t} := r\d_r u_{\ep} = r\d_r ( u * \rho_{\ep}),$$
and then we have the following Friedrichs' type estimate.

\begin{lemma}[\cite{Li23}]
\label{lem-pfc-004}
For a function $u\in \cF(B_1)$, a small number $\delta>0$ 
and a point $z\in B^*_{1-2\delta}$, we have, for all $\ep < \min\{ |z|, \delta \}$,
\begin{equation}
\label{pfc-014}
| r\d_r (u * \rho_{\ep})(z) - r( \d_r u * \rho_{\ep})(z) | \leq 2\ep ||\nabla u ||_{L^1(B_{1-\delta})}.
\end{equation}

\end{lemma}

It is a standard fact that a function $u\in \cF(B_1)$ is in the Sobolev space 
$W^{1,p}_{loc}(B_1)$ for any $1\leq p< 2$.
Hence by Lemma \ref{lem-pfc-004} we can compare  with the following two convolutions as 
\begin{equation}
\label{pfc-015}
r(\d_r u)_{\ep} (z) = \dot{u}_{\ep, t} + O(\ep),
\end{equation}
on any relatively compact domain $\Omega\subset \subset B_1^*$.
Therefore, Proposition \ref{prop-pfc-001} implies that
 the regularization converges on $\Omega$ as 
$ \dot{u}_{\ep, t} \rightarrow \d_t u$
strongly in $L^p$-norm for any $p\geq 1$.
Thanks to the slicing theory again,
we can infer the following convergence results.
\begin{prop}[\cite{Li23}]
\label{prop-pfc-002}
For a function $u\in \cF(B_1)$, and any constants $1< A <B$, 
there exists a subsequence $\{ u_{\ep_k}\}$ of its standard regularization
satisfying 
$ I(u_{\ep_k, t}) \rightarrow I(u_t), $
as $k \rightarrow +\infty$ for almost all $t\in [-B, -A]$.
\end{prop}

We end up this section with another application of Proposition \ref{prop-pfc-001}
and Lemma \ref{lem-pfc-004}, 
and it will provide the a priori estimate on the maximal directional Lelong numbers of the regularization. 

\begin{lemma}[\cite{Li23}]
\label{lem-pfc-005}
Fixing any two constants $1< A < B$, 
there exists a uniform constant $C>0$ such that we have 
\begin{equation}
\label{pfc-016}
M_B(u_{\ep}) \leq 2 M_A(u) + C\ep,
\end{equation}
for all $\ep <  \ep_0:= \frac{1}{2} \min\{ (e^{-A} - e^{-B}), e^{-B} \}. $
\end{lemma}

\section{The K\"{a}hler cone structure}
\smallskip

In this section, we are going to use a different point 
of view to look at the complex hessian of a function $u\in \cF^{\infty}(B_1)$.
In fact, there is a natural K\"ahler cone structure on the space 
$(\bC^{n+1})^* \cong (\bR^{2n+2})^*$, 
that induces the standard Sasakian structure on the unit sphere $S^{2n+1}$.
In the following, we will decompose the usual complex structure on $\bC^{n+1}$
with respect to this K\"ahler cone structure.

Let $\bR^{2n+2}$ be the $(2n+2)$-dimensional Euclidean space 
with rectangular coordinates 
$$ (x^0,\cdots, x^{n}, y^0, \cdots, y^{n}), $$
and we briefly write it as $(x^A, y^A)$ for $A = 0, 1, \cdots, n$. 
Then the $(2n+1)$-dimensional hypersphere $S_r$ with radius $r$ is defined by
$$\sum_A  \left\{ (x^A)^2 + ( y^A )^2 \right\} = r^2, $$ 
and denote $S^{2n+1}$ by the unit sphere with $r=1$. 
Put 
$$ z^A: = x^A + i y^A, $$
and then $z^A$'s define a complex structure in $\bR^{2n+2}$. 
Its standard almost complex structure is given by the $(1,1)$-tensor field
$$ I: = \sum_A \left(  \frac{\d}{\d {y^A}} \otimes  dx^A - \frac{\d}{\d{x^A}} \otimes dy^A    \right) = \sum_A i \left(  \frac{\d}{\d z^A} \otimes dz^A -   \frac{\d}{\d \bar{z}^A} \otimes d\bar{z}^A     \right). $$

Denote $g$ by the flat metric on $\bR^{2n +2}$, 
and then its associated K\"ahler form on $\bC^{n+1}$ is defined as 
$$ \omega_e: = \frac{i}{2} \ddbar r^2 = \frac{i}{2} \sum_A  dz^A \wedge d\bar{z}^A.$$

\subsection{Sasakian manifolds}
To talk about
a \emph{Sasakian structure} on $S^{2n+1}$, 
it is equivalent to describe a \emph{K\"ahler cone structure}
on the product space $(\bR^{2n+2})^*  \cong S^{2n+1} \times \bR_+ $, see Chapter 6, \cite{BG08}.
First we note that the flat metric $g$ splits as a metric cone as 
$$ g = dr^2 + r^2 g_0, $$
for any radius $r >0$. 
Here $g_0$ is the canonical metric on $S^{2n+1}$ with constant sectional curvature $1$,
and it is the restriction of $g$ to the sphere. 
Denote $\eta_0$ by the contact $1$-form: 
\begin{eqnarray}
\label{kc-001}
 \eta_0: &=&  I (r^{-1} dr)
\nonumber\\
&=& \frac{1}{r^2} \sum_A \left( y^A dx^A - x^A dy^A \right)
\nonumber\\
&=& -\frac{i}{2r^2} \sum_A \left( z^A d\bar{z}^A - \bar{z}^A dz^A \right).
\end{eqnarray}
Then it is clear that we have 
$$ \omega_e = - d(r^2 \eta_0)/2. $$ 
Based on this K\"ahler structure on the metric cone, 
we say that the quadruple 
\begin{equation}
\label{kc-0011}
\left( S^{2n+1} \times \bR_+ , dr^2 + r^2 g_0, -d(r^2 \eta_0)/2, I \right)
\end{equation}
defines a K\"ahler cone structure on the manifold $(\bR^{2n+2})^*$.   
Moreover, 
there is another important ingredient,
the \emph{Reeb vector field} $\xi_0$, that is defined as 
\begin{eqnarray}
\label{kc-002}
 \xi_0: &=& -I (r \d_r)
\nonumber\\
&=& \sum_A \left( y^A \frac{\d}{\d x^A} - x^A \frac{\d}{\d y^A} \right)
\nonumber\\
&=& (-i)  \sum_A \left( z^A \frac{\d}{\d z^A} - \bar{z}^A \frac{\d}{\d \bar{z}^A}\right).
\end{eqnarray}
It is a holomorphic Killing field on $(\bR^{2n+2})^*$, 
and $\eta_0$ is its dual $1$-form. 
Moreover, the metric $g$ has homothetic degree two, and 
the almost complex structure $I$ has homothetic degree zero in the following sense:    
$$ \mathcal{L}_{r\d_r} g = 2g; \ \ \ \ \mathcal{L}_{r\d_r} I =0. $$

Denote $(\eta_0, \xi_0)$ also by their restrictions to the manifold $S^{2n+1}$,
and then we have the following facts with the metric $g_0$: 
\begin{enumerate}
\item[(i)] 
$\eta_0$ is a contact $1$-form, and $\xi_0$ is a Killing vector field on $S^{2n+1}$; 

\item[(ii)]
$\eta_0 (\xi_0) = 1$ and  $\  \iota_{\xi_0}  d\eta_0 (\cdot) = d\eta_0 (\xi_0, \cdot) =0 $;

\item[(iii)]
the integral curves of $\xi_0$ are exactly the great circles on $S^{2n+1}$.

\end{enumerate}

It follows that the Reeb vector field $\xi_0$ defines a \emph{ regular foliation }
$\cF_{\xi_0}$
of $S^{2n+1}$ by the great circles,
and it is nothing but the Hopf-fiberation: 
\\

\begin{tikzcd}
\ \ \ \ \ \ \ \ \ \ \ \ \ \ \ \ \ \ \ \ \ \ \ S^1 \arrow[r, hook] & S^{2n+1} \arrow[r, "p"] & \bC\bP^n.
\end{tikzcd}
\\

Let  $L_{\xi_0}$ be the trivial line bundle generated by $\xi_0$,
and then we have a splitting of the tangent space of the sphere as 
$$ TS^{2n+1} = L_{\xi_0} \oplus \cD, $$
where the contact sub-bundle  $\cD: = \ker(\eta_0)$ is the kernel of the contact $1$-form. 
Moreover, it can be identified with the the normal bundle $\nu(\cF_{\xi_0})$
of the foliation via an isomorphism induced from the metric $g_0$.

Next we define an endomorphism of $TS^{2n+1}$
by restricting the almost complex structure $I$ to $\cD$, 
and extending it trivially to $L_{\xi_0}$. 
Explicitly, it is a $(1,1)$-tensor field as 
\begin{eqnarray}
\label{kc-003}
 \Phi_0 &=& \sum_{A, B} \left\{    (x^A x^B - \delta^{AB}) \frac{\d}{\d x^A} \otimes d y^B  + ( \delta^{AB} - y^A y^B) \frac{\d}{\d y^A} \otimes dx^B \right\} 
 \nonumber\\
 &+& \sum_{A, B} \left\{    y^A x^B \frac{\d}{\d y^A} \otimes d y^B  - x^A y^B \frac{\d}{\d x^A} \otimes dx^B \right\}. 
 \end{eqnarray}
Observe that $\eta_0 \circ \Phi_0 = 0$, 
and then we can infer the following equation 
\begin{equation}
\label{kc-005}
\Phi_0^2 = - \mathbb{I} + \xi_0\otimes \eta_0. 
\end{equation}
That is to say, 
the restriction $\Phi|_{\cD}$ defines 
an almost complex structure on $\cD$,
and it is compatible with the symplectic form $d\eta_0$
in the following sense: 
\begin{equation}
\begin{split}
&  d\eta_0 (\Phi_0 X, \Phi_0 Y) = d\eta_0 (X, Y) \ \text{for all $X, Y \in \Gamma(\cD)$}; \\
&  d\eta_0 (\Phi_0 X, X) >0 \ \text{for all $X \neq 0$}. 
\end{split}
\end{equation}
It follows that 
the pair $( \cD, \Phi_0|_{\cD} )$ 
defines an almost $CR$-structure,
and 
its Levi form $L_{\eta_0}: = d\eta_0 \circ (\Phi_0 \otimes \mathbb{I} )$
is strictly pseudo-convex. 
Then it 
induces a Riemannian metric on the distribution transversal to $\xi_0$, i.e. 
we define the transversal metric as 
\begin{equation}
\label{kc-0055}
g^T(X, \Phi_0Y): = d\eta_0(X, Y)\ \ \ \text{for all} \ X, Y\in \Gamma(\cD).
\end{equation}

The upshot is that the metric $g_0$ on the sphere 
is compatible with the \emph{almost contact structure} $(\xi_0, \eta_0, \Phi_0)$ 
in the following sense: 
\begin{equation}
\label{kc-006}
g_0 = g^T + \eta_0 \otimes \eta_0. 
\end{equation}
Together with the Killing condition of $\xi_0$, 
the quadruple 
$(\xi_0, \eta_0, \Phi_0, g_0)$ 
is called a \emph{Sasakian structure} on $S^{2n+1}$. 
This means that  
the almost $CR$-structure $(\cD, \Phi_0|_{\cD})$ is in fact integrable,
and $\Phi_0$ is invariant under $\xi_0$.
Then it follows a splitting
\begin{equation}
\label{kc-007}
 \cD\otimes \bC = \cD^{1,0} \oplus  \cD^{0,1} \ \ \  \text{with} \ \ \overline{\cD^{1,0} } =\cD^{0,1},
 \end{equation}
 where $\cD^{1,0}$ and $\cD^{0,1}$
 are eigenspaces of $\Phi_0$ with eigenvalues $i$ and $-i$, respectively. 
 In particular, the eigenspace with eigenvalue $0$ is exactly $L_{\xi_0}\otimes \bC$.

Furthermore, this splitting induces a complex structure 
$\bar J$ of the normal bundle via the isomorphism 
$ (\cD, \Phi_0|_{\cD}) \cong (\nu (\cF_{\xi_0}), \bar J)$,
and then it gives a transversal holomorphic structure on the foliation.


It is clear from the construction (equation \eqref{kc-0055} and \eqref{kc-006})
that the triple $(g^T, \omega^T, \bar J)$ defines a transversal K\"ahler structure 
on the local leaf space of the foliation, where the transversal K\"ahler metric is 
$$ \omega^T: = -d\eta_0. $$

Since the foliation is regular, 
 the push forward of this transversal K\"ahler structure 
 to the base manifold (via the fiber map of the Hopf-fiberation) 
is exactly the K\"ahler structure on $\bC\bP^n$ with the usual Fubini-Study metric. 
For further discussion, the reader is referred to \cite{SH62}, \cite{BG08}, \cite{HS16} and \cite{GKN00}.





\subsection{The complex Hopf-coordinate}

In order to illustrate the Sasakian structure on the sphere in an explicit way, 
we will invoke a particular local coordinate system on the cone 
$(\bR^{2n+2})^* \cong (\bC^{n+1})^*$. 
In fact, it is a generalization of the 
 complex Hopf-coordinate in $\bC^2$, see \cite{Li23}.

First consider the following holomorphic functions on the set $\{z^0 \neq 0 \}$:
$$ \z^{\a} : = \frac{z^{\a}}{z^0} = |\z^{\a}| e^{i\vp_{\a}} $$
for all $\a: = 1, \cdots, n$,
and $\vp_{\a}$ is an argument of $\z^{\a}$ as a complex number.
Then the complex Hopf-coordinate on $(\bR^{2n+2})^*$ is introduced as 
$$ (r, \theta, \z, \bar\z): = (r, \theta, \z^1, \cdots, \z^n, \bar{\z}^1, \cdots, \bar{\z}^n), $$
for all $r\in \bR_+$, $\theta\in \bR$ and $\z^{\a}\in \bC$ with the following change of variables:

\begin{equation}
\label{hc-000}
z^0 = r e^{\frac{i}{2} \theta } \frac{\varrho(\z, \bar{\z})}{\left(1+ \sum_{\b} |\z^{\b}|^2 \right)^{1/2}}; \ \ \
z^{\a} = r e^{\frac{i}{2} \theta} \frac{\z^{\a} \cdot \varrho(\z, \bar{\z})}{\left(1+ \sum_{\b} |\z^{\b}|^2 \right)^{1/2}},
\end{equation}
where the factor $\varrho$ is defined as 
$$ \varrho(\z,\bar\z):=  \prod_{\a =1}^n \left(  \frac{\bar{\z}^{\a}}{|\z^\a|}\right)^{\frac{1}{2}}. $$
This factor is introduced for the purpose to 
gain more symmetry while taking change of variables.
For example, it enables us to write 

\begin{equation}
\label{hc-001}
z^0 
=  \frac{r e^{\frac{i}{2}  \left( \theta - \sum_{\a} \vp_{\a}  \right)     }}{\left(1+ \sum_{\b} |\z^{\b}|^2 \right)^{1/2}};
\ \ \ \ 
z^{\a} 
=  \frac{r e^{\frac{i}{2}  \left( \theta + \vp_{\a} - \sum_{\b} \widehat{\vp_{\a} } \right)     }}{\left(1+ \sum_{\b} |\z^{\b}|^2 \right)^{1/2}},
\end{equation}
where 
the notation $\sum_{\gamma} \widehat{\vp_{\a}}$ means taking the summation of all $\vp_\gamma$'s without the angle $\vp_\a$. 

In fact, 
the complex Hopf-coordinate 
is induced from a local trivialization of $(\bC^{n+1})^*$
as a principle $\bC^*$-bundle over $\bC\bP^n$.
That is to say, there is a homeomorphism of the fiber map $p$ as 
$$ \Psi: \bC^* \times  U \rightarrow p^{-1}(U), $$
that sends 
$$ ( re^{i\theta/2},  \z^{1},\cdots, \z^n )  \rightarrow (z^0,\cdots, z^n) $$
via the map defined in equation \eqref{hc-000}. 
Here $U\subset \bC^n$ is an open set on which one branch of the factor $\varrho$ is well-defined,
and it can  be identified with an open subset in $\bC\bP^{n} - [0: z^1: \cdots: z^n]$.

With the aid of this complex Hopf-coordinate, 
we can compute the contact $1$-form  in the cone $(\bR^{2n+2})^*$.
However, it is convenient to introduce the following notations first:
\begin{equation}
\label{hc-0015}
J: = -I/2 \ \ \  \text{and} \ \ \  \eta: = J(r^{-1}dr) = -\eta_0/2.
\end{equation}
This is because of the normalization in the $d^c$-operator, since we have 
$$d^c = \frac{i}{2}(\dbar - \d)=  Jd.$$
Writing the complex variable $z^A$ in the polar coordinate as 
$z^A: = r^A e^{i \theta_A}$,
it follows from equation (\ref{hc-001}) 
\begin{equation}
\label{hc-002}
\theta_0 = \theta - \sum_{\a =1}^n \vp_{\a}; \ \ \ \  \theta_{\a} = \theta + \vp_{\a} - \sum_{\gamma=1}^n \widehat{\vp_{\a}}.
\end{equation}
Hence it follows 
\begin{equation}
\label{hc-003}
\eta = \frac{1}{2r^2} \sum_A \frac{ x^A dy^A - y^A dx^A}{(x^A)^2 + (y^A)^2} \cdot (r^A)^2 = \frac{1}{2r^2}\sum_A (r^A)^2 d\theta_A,
\end{equation}
and we have 
\begin{eqnarray}
\label{hc-004}
&&2 \sum_A (r^A)^2 d\theta_A 
\nonumber\\
&=& \frac{r^2}{ 1+ \sum_{\b} |\z^{\b}|^2}\left\{ d\theta - \sum_{\a} d\vp_{\a}
+ \sum_{\a} |\z^{\a}|^2 \left( d\theta + d\vp_{\a} - \sum_{\gamma} \widehat{d\vp_{\a}} \right)         \right\}
\nonumber\\
&=& r^2 \left\{   d\theta - \sum_{\a} \left( 1 - \frac{2|\z^{\a}|^2}{1+ \sum_{\b} |\z^{\b}|^2} \right) d\vp_{\a} \right\}.
\nonumber\\
\end{eqnarray}
Observe that we can write 
$$ d\vp_{\a} = \Im\left( \frac{d\z^{\a}}{\z^{\a}}\right). $$
Therefore, if we put 
$$\cos\k_{\a}: =  1 - \frac{2|\z^{\a}|^2}{1+ \sum_{\b} |\z^{\b}|^2}\in [-1,1],$$
then the formula in equation \eqref{hc-004}
can be reduced to the following form, cf. Lemma 4.1, \cite{Li23}. 

\begin{lemma}
\label{lem-hc-001}
The normalized contact $1$-form can be written under the complex Hopf-coordinate as 
\begin{equation}
\label{hc-005}
\eta = \frac{1}{4} \left\{ d\theta - \sum_{\a} \cos\k_{\a} \cdot \Im\left( \frac{d\z^{\a}}{\z^{\a}}\right) \right\}. 
\end{equation}
\end{lemma}
Moreover, one can also directly check the following useful identities:
\begin{equation}
\label{hc-0055}
 \frac{\d\bar z^{0}}{\d\z^{\b}} = \frac{\bar z^{0}}{4\z^{\b}} \cos\k_{\b};\ \ \ \  \frac{\d\bar z^{\a}}{\d\z^{\b}} = \frac{\bar z^{\a}}{4\z^{\b}} \cos\k_{\b},
\end{equation}
for each $\a, \b= 1,\cdots, n$.
Next we can introduce a \emph{local basic function} $h$
with respect to the complex Hopf-coordinate as 
$$h(\z, \bar\z): = \log\left( 1+ \sum_{\a} |\z^{\a}|^2 \right) - \sum_{\a} \log |\z^{\a}|, $$
and then obtain a \emph{local defining equation} of  the contact $1$-form as 
\begin{equation}
\label{hc-006}
\eta = \frac{1}{4} \left\{   d\theta - i \left( \d_{\z} h - \dbar_{\z} h  \right) \right\},
\end{equation}
where the operators are defined by 
$$ \d_{\z}h : = \sum_{\a} \frac{\d h }{ \d\z^{\a} }d\z^{\a};\ \ \ \  \dbar_{\z}h : = \sum_{\a} \frac{\d h }{ \d\bar\z^{\a} }d\bar\z^{\a}. $$
Finally it follows 
\begin{eqnarray}
\label{hc-007}
d \eta &=& \frac{1}{2} dd^c_{\z} h
\nonumber\\
&=& \frac{1}{2} dd^c_{\z} \log\left( 1+ \sum_{\a} |\z^{\a}|^2 \right) 
\nonumber\\
&=& \omega_{FS}. 
\end{eqnarray}
As we have expected, 
the transversal K\"ahler metric $d\eta$ is exactly the Fubini-Study metric $\omega_{FS}$ 
on the quotient space $\bC\bP^n$, and we have its total volume 
$$\int_{\bC\bP^n} \omega^n_{FS} = \pi^n. $$

Although everything is computed under the complex Hopf-coordinate, 
we emphasis that the local defining equation of $\eta$ (equation (\ref{hc-006})), and hence equation (\ref{hc-007}),
actually follow from the Sasakian structure.
This means that they are in fact independent of the chosen 
holomorphic coordinate on $\bC\bP^n$ (induced from a local trivialization), 
possibly with a different basic function $h$ and a re-normalized angle $\theta$, see \cite{GKN00}, \cite{GZ12}.

\begin{rem}
\label{rem-hc-001}
There is another local trivialization of the fiber map $p$ as
$$ \Psi':  \bC^*\times \bC^n \rightarrow p^{-1}(\bC^n),$$
that sends $(r, \theta', \z, \bar\z)$ to
$$ z^0: = \frac{re^{i\theta'}}{\left( 1+ \sum_{\b}|\z^{\b}|^2 \right)^{1/2} }; \ \ \  
z^{\a} = \frac{re^{i\theta'} \cdot \z^\a}{\left( 1+ \sum_{\b}|\z^{\b}|^2 \right)^{1/2} }. $$
Thus we have a different holomorphic coordinate 
 $\z\in\bC\bP^n$ that is induced from the trivialization $\Psi'$.
Moreover, the local basic function $h'$ corresponding to $\Psi'$ is defined as 
$$ h'(\z,\bar\z): = \log\left( 1+ \sum_{\a} |\z^{\a}|^2 \right). $$

For an $S^1$-invariant function, 
these two trivializations make no difference on its value 
since they only differ by an angle factor, see Remark 3.1, \cite{Li23}. 
Therefore, we also refer $\Psi'$ as a complex Hopf-coordinate.

\end{rem}

\subsection{The complex structure splits}
\label{sec-cs}
As we have seen in equation \eqref{kc-007},
the complex structure on the cone $(\bR^{2n+2})^*$
splits with respect to the Sasakian structure on the sphere. 
Moreover, this splitting can be explicitly written down 
locally through the complex Hopf-coordinate as follows. 

It turns out that 
it is easier to first describe it on the cotangent bundle, 
namely, 
we consider  the $(1,0)$-part of the complexified cotangent bundle of the cone as 
$$(T^*)^{1,0}(\bR^{2n+2})^* \subset T^{*}(\bR^{2n+2})^* \otimes \bC. $$
Then there is a well-defined $1$-form on the cone as 
\begin{equation}
\label{cs-000}
\lambda^0:= dr - i r\eta_0; \ \ \ \  \bar{\lambda}^0:= dr + i  r\eta_0,
\end{equation}
and it satisfies
$$ I(\lambda^0) = i \lambda^0; \ \ \ \  I(\bar{\lambda}^0) = (-i)\bar{\lambda}^0. $$
Locally we introduce the following
$1$-forms 
under the complex Hopf-coordinate: 
\begin{equation}
\label{cs-001}
\lambda^{\a}:= d\z^{\a}; \ \ \ \ \bar{\lambda}^{\a}: = d\bar\z^{\a},
\end{equation}
for all $\a=1,\cdots, n$. 
Since $\z^{\a}$'s are holomorphic functions,
it follows 
$$ I(\lambda^{\a}) = i \lambda^{\a}; \ \ \ \ I(\bar{\lambda}^{\a}) = (-i) \bar{\lambda}^{\a}.$$
Then it is clear that the $n$-tuples 
$ \{ \lambda_1,\cdots, \lambda_n \} $
is a local coframe field of the bundle $(T^*)^{1,0}(\bC\bP^n)$ over the base manifold. 
Thanks to the local defining equation of $\eta$ (equation \eqref{hc-006}), 
we can further infer that the $(n+1)$-tuples 
$$ \{\lambda_0, \lambda_1,\cdots, \lambda_n \} $$
builds a local coframe field of the bundle $(T^*)^{1,0}(\bR^{2n+2})^* $ over the cone,
and its complex conjugate is a local coframe of the bundle  $(T^*)^{0,1}(\bR^{2n+2})^* $.
Clearly, these coframes give a decomposition of the complex structure 
with respect to the K\"ahler cone structure. 
Moreover, this construction is due to the Sasakian structure, 
i.e. it is independent of the chosen holomorphic coordinate on $\bC\bP^n$,
because the local defining equation of $\eta$ is.

On the other hand, we can also describe  the splitting in the dual space:
\begin{equation}
\label{cs-0019}
T(\bR^{2n+2})^* \otimes \bC = \left( L_{e_0} \oplus L_{\bar{e}_0}\right) \oplus \cD^{1,0} \oplus \cD^{0,1},
\end{equation}
where the vector fields $e_0$ and $\bar{e}_0$ are defined as 
$$ e_0: = \frac{1}{2}\left\{ \d_r + i r^{-1}\xi_0 \right\};\ \ \ \ \bar{e}_0: =  \frac{1}{2}\left\{ \d_r - i r^{-1}\xi_0 \right\}. $$
In fact,  the Reeb vector field can written as 
$$\xi_0 = -I(r\d_r) = -2 \frac{\d}{\d\theta}, $$
and then it is clear that we have 
$$ I(e_0) = i e_0;\ \ \ \ I(\bar{e}_0) = (-i)\bar{e}_0. $$
Moreover, we define the following vector fields 
under the complex Hopf-coordinate:
$$ e_{\a}:= \d_{\z^{\a}} + \frac{i}{4\z^{\a}} \cos\k_{\a}\  \xi_0;\ \ \ \    \bar{e}_{\a}:= \d_{\bar\z^{\a}} - \frac{i}{4\bar\z^{\a}} \cos\k_{\a}\ \xi_0.   $$
Then one can check 
$$ \lambda^A(e_B) = \delta^A_B,$$
for all $A,B= 0,\cdots, n$. 
In particular,  the $e_{\a}$'s, and hence $\bar{e}_{\a}$'s,  belong to
the complexified contact   sub-bundle $\cD\otimes\bC$. 
Moreover, it follows from equation (\ref{hc-0055}) that we have 
\begin{equation}
\label{cs-002}
I(e_{\a}) = i e_{\a}; \ \ \ \  I(\bar{e}_{\a}) = (-i) \bar{e}_{\a}.
\end{equation}
It follows that the almost complex structure $\Phi_0$ can be written locally as 
$$ \Phi_0 = i \sum_{\a} e_{\a}\otimes d\z^{\a} - i \sum_{\a} \bar{e}_{\a}\otimes d\bar\z^{\a}, $$
and then we have 
\begin{equation}
\label{cs-003}
\Phi_0(e_{\a}) = i e_{\a}; \ \ \ \  \Phi_0(\bar{e}_{\a}) = (-i) \bar{e}_{\a}. 
\end{equation}

Therefore, the $n$-tuples $\{e_1,\cdots, e_n\}$
is a local frame of the bundle $\cD^{1,0}$.
It follows that the $(n+1)$-tuples $\{e_0, e_1,\cdots, e_n\}$
builds a local frame of the bundle $T^{1,0}(\bR^{2n+2})^*$.
Together with its complex conjugate, they describe the splitting (equation \eqref{cs-0019})
of the complex structure on the tangent space of the cone.

\bigskip

\section{The decomposition formula}
\smallskip

In this section, we will compute the 
following integral, for a function $u\in \cF^{\infty}(B_1)$:
$$\int_{S_r} d^c u \wedge (dd^c u)^n, $$
on the hypersphere $S_r\subset \bR^{2n+2}$ with radius $r\in(0,1)$. 
Consider the space $(\bR^{2n+2})^* \cong (\bC^{n+1})^*$
as a K\"ahler cone over the Sasakian manifold $S^{2n+1}$, 
and denote $(\xi_0, \eta_0, \Phi_0, g_0)$ by its standard Sasakian structure. 
We recall the concept of basic forms. 

\begin{defn}
\label{def-df-001}
A $k$-form $\vartheta$ on  $S^{2n+1}$ is called basic if it satisfies 
$$ \iota_{\xi_0} \vartheta =0;\ \ \ \   \mathcal{L}_{\xi_0}\vartheta =0.  $$
\end{defn}

The exterior differential preserves basic forms. 
There is a natural splitting of the complexification of the bundle of the basic $k$-forms ${\bigwedge}_B^{k}(S^{2n+1})$ as 
$$ {\bigwedge}^k_B    (S^{2n+1}) \otimes \bC = \bigoplus_{i+j =k}  {\bigwedge}_B^{i, j} (S^{2n+1}),$$
where $ {\bigwedge}_B^{i, j} (S^{2n+1})$  
denotes the bundle of basic forms of type $(i, j)$. 
Then one can define the operators $\d_B$ and $\dbar_B$. 
Put 
$$d_B: = d\ |_{{\bigwedge}^k_B }; \ \ \ \ d^c_B:= \frac{i}{2} (\dbar_B - \d_B), $$ 
and it follows as usual 
$$ d_B = \d_B + \dbar_B; \ \ \  d_B d^c_B = i\d_B \dbar_B; \ \ \  (d_B)^2 = (d^c_B)^2=0. $$
In other words, 
the operators $\d_B$ and $\dbar_B$ are exactly (the pull back of) 
the $\d$ and $\dbar$ operators on $\bC\bP^n$
under the isomorphism with the transversal holomorphic structures: $ (\nu(\cF_{\xi_0}), \bar J)\cong (\cD, \Phi_0|_{\cD})$.

\subsection{Computations}
It is clear that $u$ is basic, 
if it is in the family $\cF^{\infty}(B_1)$. 
Then we can decompose its exterior derivative as 
$$du= (u_{r}) dr + d_B u, $$
and hence it follows  
\begin{eqnarray}
\label{df-004}
d^c u 
&=&  (u_r) J dr + J d_B u
\nonumber\\
&=& (r u_r) \eta + d^c_B u,
\end{eqnarray} 
where $\eta = -\eta_0/2$ is the normalized contact $1$-form. We compute
\begin{eqnarray}
\label{df-005}
dd^c u 
&=&   d \left\{ (r u_r) \eta + d^c_B u \right\} 
\nonumber\\
&=& (r  u_r) d\eta  + (ru_r)_{,r} dr\wedge \eta + d_B(ru_r) \wedge \eta + dr\wedge \d_r(d^c_B u)+ d_B d^c_B u
\nonumber\\
&=&  \Theta_1 + \Theta_2,
\end{eqnarray} 
where the $2$-forms are defined as 
\begin{align}
\label{df-006}
\Theta_1: = (ru_r)_{,r} dr\wedge\eta + dr\wedge \d_r (d^c_B u) + d_B(ru_r)\wedge \eta;
\\
\Theta_2:= (ru_r) d\eta + d_B d^c_B u.
\end{align} 
Then we first come up with 
the following decomposition formula of the complex hessian of $u$ restricted to the sphere.
\begin{lemma}
\label{lem-df-001}
For a function $u\in \cF^{\infty}(B_1)$, we have
\begin{eqnarray}
\label{df-007}
dd^c u|_{S_r} &=& (ru_r) d\eta + d_B (ru_r) \wedge \eta + d_B d^c_B u
\nonumber\\
&=& \Theta_2 +  d_B (ru_r) \wedge \eta,
\end{eqnarray}
on each hypersphere $S_r$ with $r\in (0,1)$. 
\end{lemma} 

Next we can compute the $(n, n)$-form on the hypersphere as 
\begin{eqnarray}
\label{df-008}
( dd^c u|_{S_r} )^n 
&=& \left(  \Theta_2 +  d_B (ru_r) \wedge \eta   \right)^n
\nonumber\\
&=& \Theta_2^n + n \Theta_2^{n-1} \wedge d_B (ru_r) \wedge \eta.
\end{eqnarray}
  
Since the two terms in $\Theta_2$ are both of type $(1,1)$ in the transversal K\"ahler structure, 
we  expand its $n$th power as 
\begin{eqnarray}
\label{df-009}
\Theta_2^n &=&  (ru_r)^n (d\eta)^n  +  (d_B d^c_B u)^n
\nonumber\\
&+&  \sum_{k=1}^{n-1} \left( \begin{array}{c} n \\ k \end{array} \right) (ru_r)^{n-k} (d\eta)^{n-k} \wedge (d_B d^c_{B} u)^k.
\end{eqnarray}
Observe that we have for each $k=0,\cdots, n$
\begin{equation}
\label{df-010}
d^c_B u \wedge (d\eta)^{n-k} \wedge (d_B d^c_B u)^k =0.
\end{equation}
It follows 
\begin{eqnarray}
\label{df-011}
&&d^c u\wedge (dd^c u|_{S_r})^{n}
\nonumber\\
&=& \left\{  (ru_r) \eta + d^c_B u   \right\} \wedge \left\{ \Theta_2^n + n \Theta_2^{n-1} \wedge d_B (ru_r) \wedge \eta \right\}
\nonumber\\
&=& (ru_r) \eta\wedge \Theta_2^n + n \ d^c_B u \wedge \Theta_2^{n-1}  \wedge d_B (ru_r) \wedge \eta.
\end{eqnarray}
The first term on the R.H.S. of \eqref{df-011} can be computed as
\begin{eqnarray}
\label{df-012}
(ru_r) \eta \wedge \Theta_2^n &=&  (ru_r)^{n+1} \eta\wedge (d\eta)^n  + (ru_r) \eta \wedge  (d_B d^c_B u)^n
\nonumber\\
&+&  \sum_{k=1}^{n-1} \left( \begin{array}{c} n \\ k \end{array} \right) (ru_r)^{n-k+1} \eta\wedge (d\eta)^{n-k} \wedge (d_B d^c_{B} u)^k,
\end{eqnarray}
and the second term is 
\begin{eqnarray}
\label{df-013}
&& (-n)  \ \eta \wedge d_B (ru_r) \wedge d^c_B u \wedge \Theta_2^{n-1} 
\nonumber\\
&=& (-n) \  \eta \wedge d_B(r u_r) \wedge d^c_B u \wedge \sum_{j=1}^{n} \left( \begin{array}{c} n-1 \\ j-1 \end{array} \right)  (ru_r)^{n-j} (d\eta)^{n-j} 
\wedge (d_B d^c_B u)^{j-1}
\nonumber\\
\end{eqnarray}
In conclusion, we have the following formula. 
\begin{lemma}
\label{lem-df-002}
For a function $u\in\cF^{\infty}(B_1)$, we have the expansion formula 
\begin{eqnarray}
\label{df-014}
&&d^c u \wedge (dd^c u|_{S_r})^n
\nonumber\\
&=& \eta \wedge \left\{ \ 
\sum_{k=0}^{n} \left( \begin{array}{c} n \\ k \end{array} \right) (ru_r)^{n-k+1} (d\eta)^{n-k} \wedge (d_B d^c_{B} u)^k \right.
\nonumber\\
&-&  \left. n \ 
 d_B(r u_r) \wedge d^c_B u \wedge \sum_{k=1}^{n} \left( \begin{array}{c} n-1 \\ k-1 \end{array} \right) 
  (ru_r)^{n-k} (d\eta)^{n-k}  \wedge (d_B d^c_B u)^{k -1} \right\},
\nonumber\\
\end{eqnarray}
\end{lemma}

Observe that all the terms in the bracket on the R.H.S. of equation (\ref{df-014})
has transversal degree $2n$. 
Then their wedge product with the contact $1$-form $\eta$
is equivalent to taking the wedge with $d\theta /4$.
Thanks to equation (\ref{hc-005}) and (\ref{hc-006}),
we can write down this $(2n+1)$-form under the complex Hopf-coordinate as 

\begin{eqnarray}
\label{df-015}
&&d^c u \wedge (dd^c u|_{S_r})^n
\nonumber\\
&=& \frac{1}{4} d\theta \wedge \left\{ \ 
\sum_{k=0}^{n} \left( \begin{array}{c} n \\ k \end{array} \right) (ru_r)^{n-k+1} \omega_{FS}^{n-k} \wedge ( d_B d^c_B u)^k \right.
\nonumber\\
&-&  \left. n \ 
  d_B (r u_r) \wedge  d^c_B u \wedge \sum_{k=1}^{n} \left( \begin{array}{c} n-1 \\ k-1 \end{array} \right) 
 (ru_r)^{n-k} \omega_{FS}^{n-k}  \wedge ( d_B d^c_B u)^{k -1} \right\}.
\nonumber\\
\end{eqnarray}
Here the operators $d_B$ and $d^c_B$ should be interpreted as 
the $d$ and $d^c$ operators on the complex projective space $\bC\bP^n$.

Recall that we have taken a change of variables $t: = \log r$,
and denote the function $u$ under this new variable as  
$$ u_t (\z): = \hat{u}(t, \z, \bar\z) = u(e^t, \z, \bar\z), $$
and then it follows 
$$ \dot{u}_t (\z):= \d_t \hat{u}(t, \z, \bar\z) = r\d_r u(r,\z,\bar\z). $$
Finally, we obtain the following decomposition formula. 

\begin{theorem}
\label{thm-df-001}
For a function $u\in\cF^{\infty}(B_1)$, we have 
\begin{eqnarray}
\label{df-016}
\frac{1}{\pi} \int_{S_r} d^c u \wedge (dd^c u)^n 
= \sum_{k=0}^n   \begin{pmatrix} n+1 \\ k\end{pmatrix} 
\int_{\bC\bP^n} 
(\dot{u}_t)^{n+1-k} \omega_{FS}^{n-k}\wedge (d_B d^c_B u)^k,
\nonumber\\
\end{eqnarray}
on each hypersphere $S_r$ with radius $ r= e^t\in (0,1)$. 
\end{theorem}
\begin{proof}
Take integration on both sides of equation (\ref{df-015}), 
and we can first integrate out the $d\theta$-direction over $[0, 4\pi]$
due to Fubini's Theorem. 
Then perform the integration by parts on the second term 
in the bracket on the R.H.S. of this equation as 
\begin{eqnarray}
\label{df-017}
&& (-n) \int_{\bC\bP^n} \sum_{k=1}^{n}  \left( \begin{array}{c} n-1 \\ k-1 \end{array} \right) 
(\dot{u}_t)^{n-k} d_B \dot{u}_t \wedge d^c_B u \wedge \omega_{FS}^{n-k} \wedge (d_B d^c_B u)^{k-1}
\nonumber\\
&=&  \int_{\bC\bP^n} \sum_{k=1}^{n}  \frac{n}{n-k+1}
\left( \begin{array}{c} n-1 \\ k-1 \end{array} \right) 
(\dot{u}_t)^{n-k+1} \omega_{FS}^{n-k} \wedge (d_B d^c_B u)^k. 
\nonumber\\
\end{eqnarray}
Plug in the first term in the bracket, and we obtain 

\begin{eqnarray}
\label{df-018}
&&  \left\{ \ \sum_{k=0}^{n} 
 \left(\begin{array}{c} n \\ k \end{array} \right)  + \sum_{k=1}^n  \frac{n}{n-k+1}\left( \begin{array}{c} n-1 \\ k-1 \end{array} \right)  \right\}
(\dot{u}_t)^{n-k+1} \omega_{FS}^{n-k} \wedge (d_B d^c_B u)^k
\nonumber\\
&=&   \sum_{k=1}^n \left( \begin{array}{c} n+1 \\ k\end{array} \right) 
(\dot{u}_t)^{n+1-k} \omega_{FS}^{n-k}\wedge (d_B d^c_B u)^k
+  (\dot{u}_t)^{n+1} \omega_{FS}^n.
\nonumber\\
\end{eqnarray}
Here we have used the combinatorial identity: 
\begin{equation}
\label{df-019}
\left( \begin{array}{c} n \\ k \end{array} \right)  +  \frac{n}{n-k+1}\left( \begin{array}{c} n-1 \\ k-1 \end{array} \right) 
= \left( \begin{array}{c} n+1 \\ k\end{array} \right). 
\end{equation}
Finally, combine with the two terms on the R.H.S. of equation (\ref{df-018}), 
and then our result follows.

\end{proof}

The above formula (equation \eqref{df-016})
boils down to the decomposition formula in the two dimensional case
if  we take $n=1$, see Theorem 4.4, \cite{Li23}.

\subsection{The positivity}
The plurisubharmonicity of the function $u$
induces the positivity of its complex hessian on $(\bC^{n+1})^*$.
As we have seen,
there is a natural splitting of the complex structure of the cone $(\bC^{n+1})^*$
under the standard Sasakian structure of $S^{2n+1}$. 
Therefore, we can study the positivity of the complex hessian
under this splitting.

Recall that  decompose the complex hessian of 
a function $u\in \cF^{\infty}(B_1)$ in equation (\ref{df-005}) as 
$$ dd^c u = \Theta_1 + \Theta_2.$$
Here the $2$-form $\Theta_1$ can be written as 
\begin{eqnarray}
\label{ps-001}
\Theta_1 &=& (ru_r)_{,r} dr\wedge\eta + dr\wedge d^c_B u_r + d_B(ru_r)\wedge \eta
\nonumber\\
&=& r^{-1}(ru_r)_{,r} (i\lambda^0\wedge\bar\lambda^0) + \frac{1}{2} \left( \lambda^0\wedge d^c_B u_r  + \bar\lambda^0\wedge d^c_B u_r  \right)
\nonumber\\
&+& \frac{i}{2r} \left\{  \bar\lambda_0 \wedge d_B(ru_r) - \lambda_0 \wedge d_B(ru_r)  \right\}. 
\end{eqnarray}

Moreover, the $2$-form $\Theta_2$ is of degree $(1,1)$ 
under the transversal holomorphic structure. 
In particular, it descends to a $(1,1)$-form with continuous coefficients on $\bC\bP^n$ as 
$$ \Theta_2: = (\dot{u}_t) \omega_{FS} + d_B d^c_B u. $$
Then we are going to prove that this $(1,1)$-form is positive for each $t<0$. 

\begin{lemma}
\label{lem-ps-001}
The $2$-form  $\Theta_2$ is a positive $(1,1)$-current on $\bC\bP^n$, i.e. we have 
\begin{equation}
\label{ps-0015}
(\dot{u}_t) \omega_{FS} + d_B d^c_B u \geq 0, 
\end{equation}
for each $t<0$ fixed. 
\end{lemma}
\begin{proof}
Let $\gamma^{\a}$ be an arbitrary local $(1,0)$-form on $\bC\bP^n$ for $\a=2,\cdots, n$. 
Then we need to prove the following $(n,n)$-form 
$$ T: = \Theta_2 \wedge (i\gamma^2\wedge \bar\gamma^2)\wedge\cdots\wedge (i\gamma^n\wedge \bar\gamma^n)$$
is a positive measure.
Thanks to the splitting of the complexified cotangent bundle, 
each of $\gamma^{\a}$'s can be written as a linear combination of the coframe $\lambda^{\b}$
for $\b=1,\cdots, n$,
see Section \ref{sec-cs}. 
Then we can write $T$ with respect to the local coframe, and obtain 
\begin{equation}
\label{ps-002}
T = f(r, \z, \bar\z) (i \lambda^1\wedge \bar\lambda^1)\wedge\cdots\wedge   (i \lambda^n\wedge \bar\lambda^n).
\end{equation}
Then everything boils down to prove 
that the continuous function $f$ is always non-negative. 
Take $T$ as an $(n,n)$-form on the cone $(\bC^{n+1})^*$, 
and then we can wedge it with the positive $(1,1)$-current $i\lambda^0\wedge\bar\lambda^0$ as 
\begin{eqnarray}
\label{ps-003}
(i\lambda^0\wedge\bar\lambda^0)\wedge T
&=&  \left\{ ( i\lambda^0\wedge\bar\lambda^0) \wedge \Theta_2 \right\}
\wedge (i\gamma^2\wedge \bar\gamma^2)\wedge\cdots\wedge (i\gamma^n\wedge \bar\gamma^n)
\nonumber\\
&=& \left\{ ( i\lambda^0\wedge\bar\lambda^0) \wedge (\Theta_1 + \Theta_2 )\right\}
\wedge (i\gamma^2\wedge \bar\gamma^2)\wedge\cdots\wedge (i\gamma^n\wedge \bar\gamma^n)
\nonumber\\
&=& dd^c u \wedge ( i\lambda^0\wedge\bar\lambda^0) \wedge
(i\gamma^2\wedge \bar\gamma^2)\wedge\cdots\wedge (i\gamma^n\wedge \bar\gamma^n)
\nonumber\\
&\geq& 0.
\end{eqnarray}
Here we have used equation (\ref{ps-001}) on the second line 
on the R.H.S. of equation (\ref{ps-003}). 
Hence we can infer $f(r, \z, \bar\z)\geq 0$, 
and then our result follows. 

\end{proof}

\smallskip

\section{The general results}
\smallskip

The goal of this section is to provide a useful estimate on the 
complex Monge-Amp\`ere mass 
$$ \int_{B_r} (dd^c u)^{n+1} $$
on each ball $B_r$ inside the unit ball. 
It means that we need to find an upper bound 
of the decomposition formula (Theorem \ref{thm-df-001}). 
This will be finished via an induction argument.  

Recall that the maximal directional Lelong number
$M_A(u)$ at a distance $A>0$ is defined as 
the maximum of $\dot{u}_{t}(\z)$ over all $\z\in\bC\bP^n$ with fixed $t= -A$. 
It is always a non-negative and finite real number, and non-increasing in $A$ as $A\rightarrow +\infty$. 

\subsection{A priori estimates}
In the following context, the positivity condition (Lemma \ref{lem-ps-001}) will be used repeatedly. 
Therefore, it is convenient to introduce the following brief notations for the $(1,1)$-forms:
$$ \ba:= \omega_{FS}; \ \ \ \ \  \bb: = d_B d^c_B u,$$
and for the functions 
$$ \bc: =\dot{u}_t; \ \ \ \ \  M: = M_A(u), $$
for any $ t< -A$. 
Moreover, we also denote 
$$ \be: = \bc\ba + \bb = \Theta_2, $$
and then it follows from the positivity conditions
\begin{align}
\label{gr-001}
\ba >0; \ \ \ \  \be\geq 0; \ \ \ \ M \geq \bc \geq 0.
\end{align}

During the combinatorial computation, 
we will delete the symbol $\wedge$ 
between the $2$-forms $\ba$, $\bb$ and $\be$.
In fact, 
there is no harm to pretend the wedge product between them 
as an ordinary multiplication between polynomials, 
since $2$-forms are commutative with respect to the wedge.  
Then we first have a rough estimate as follows. 
\begin{lemma}
\label{lem-gr-001}
For a function $u\in\cF^{\infty}(B_1)$, we have the estimate
\begin{equation}
\label{gr-002}
\int_{B_r} (dd^c u)^{n+1} \leq (n+1)\pi^{n+1} M_A^{n+1}(u),
\end{equation}
for all $r < e^{-A}$. 
\end{lemma}
\begin{proof}
A direct computation shows that we have the following equation
\begin{eqnarray}
\label{gr-003}
&& \sum_{k=0}^n  \left( \begin{array}{c} n+1 \\ k\end{array} \right) 
M_A^{n+1-k} \omega_{FS}^{n-k}\wedge (d_B d^c_B u)^k
\nonumber\\
&-& \sum_{k=0}^n  \left( \begin{array}{c} n+1 \\ k\end{array} \right) 
(\dot{u}_t)^{n+1-k} \omega_{FS}^{n-k}\wedge (d_B d^c_B u)^k
\nonumber\\
&=& (M_A - \dot{u}_t) 
\left\{  \sum_{k=0}^n (M_A \omega_{FS} + d_B d^c_B u)^{n-k}\wedge \Theta_2^k \right\}.
\nonumber\\
\end{eqnarray}
In fact, we can infer it from the combinatorial identity 
\begin{eqnarray}
\label{gr-004}
&& \sum_{k=0}^n  \left( \begin{array}{c} n+1 \\ k\end{array} \right) 
 M^{n+1-k} \ba^{n-k} \bb^k 
 -  \sum_{k=0}^n  \left( \begin{array}{c} n+1 \\ k\end{array} \right) 
 \bc^{n+1-k} \ba^{n-k} \bb^k
 \nonumber\\
 &=& (M - \bc) \left\{ \sum_{k=0}^n (M\ba+\bb)^{n-k} \cdot (\bc\ba + \bb)^k  \right\} \geq 0.
 \nonumber\\ 
\end{eqnarray}
Thanks to the Stokes Theorem, each term like 
$$ \omega^{n-k}_{FS}\wedge (d_B d^c_B u)^k $$
for $k=1,\cdots, n$ vanishes if we take its integration on $\bC\bP^n$.
Therefore, the R.H.S. of the decomposition formula (Theorem \ref{thm-df-001}) 
can be estimated from the above as 
\begin{equation}
\label{gr-005}
\frac{1}{\pi}\int_{S_r} d^c u \wedge (dd^c u)^{n} \leq (n+1) M_A^{n+1} \int_{\bC\bP^n} \omega_{FS}^n,
\end{equation}
and then our result follows.

\end{proof}

First take the radius $r\rightarrow 0^+$ on the L.H.S. of equation (\ref{gr-002}),
and then it converges to the residual Monge-Amp\`ere mass of $u$ at the origin.
Next let the distance $A$ converges to $+\infty$ on the R.H.S. of this equation, 
and then we obtain the following rough estimate. 

\begin{theorem}
\label{thm-gr-001}
For a function $u\in \cF^{\infty}(B_1)$, we can control its residual mass 
by the maximal directional Lelong number at the origin as 
\begin{equation}
\label{gr-006}
\tau_u(0) \leq (n+1) \lambda^{n+1}_u(0).
\end{equation}
\end{theorem}

However, this rough estimate in Lemma \ref{lem-gr-001}
is not enough for our purpose to prove the zero mass conjecture, 
since there is no Lelong number that could possibly appear 
from the R.H.S. of equation (\ref{gr-002}). 
The difficulty is that we should keep at least one $\bc$, 
but no combination terms like $\bc\cdot\bb$ 
in the end of the estimate. 

In order to overcome this difficulty, 
we will utilize the positivity of 
$\ba$, $\bc$ and $\be$, and rewrite the decomposition formula as follows. 

\begin{cor}
\label{cor-gr-001}
For a function $u\in\cF^{\infty}(B_1)$, 
the decomposition formula can be written as 
\begin{eqnarray}
\label{gr-0065}
&&\frac{1}{\pi} \int_{S_r} d^c u \wedge (dd^c u)^n 
\nonumber\\
&=& \sum_{k=0}^n   \left( \begin{array}{c} n+1 \\ k+1 \end{array} \right) (-1)^{k}
\int_{\bC\bP^n} 
(\dot{u}_t)^{k+1} \omega_{FS}^{k}\wedge \Theta_2^{n-k}. 
\end{eqnarray}
on each hypersphere $S_r$ with radius $r\in (0,1)$. 
\end{cor}
\begin{proof}
It directly follows from the following combinatorial identity
\begin{eqnarray}
\label{gr-0066}
&&\sum_{k=0}^n  \left( \begin{array}{c} n+1 \\ k\end{array} \right) 
\bc^{n-k+1}  \ba^{n-k} \bb^k
\nonumber\\
&=& \sum_{k=1}^{n+1} \left( \begin{array}{c} n+1 \\ k\end{array} \right) 
(-1)^{k-1} ( \bc\ba + \bb )^{n+1-k} \bc^k \ba^{k-1}
\nonumber\\
&=& \sum_{ j=0}^n  \left( \begin{array}{c} n+1 \\  j+1   \end{array} \right) 
(-1)^j  \be^{n-j} \bc^{j+1} \ba^j.  
\end{eqnarray}
\end{proof}

The formula in Corollary \ref{cor-gr-001} gives a new representation of 
the decomposition formula of the complex Monge-Amp\`ere mass, 
such that each monomial in it like $\be^{n-j} \bc^{j+1} \ba^j$ is positive. 
Then we can infer the following crucial estimate on the upper bound of such 
monomials on the R.H.S. of equation (\ref{gr-0065}). 
Before moving on, we introduce the notation
$$ \{  \ba, \bb  \}^n_1. $$
It stands for a polynomial consisting of 
monomials like $\ba^{n-k} \cdot \bb^k$ for $k=1, \cdots, n$, 
and here we permit the constant $M$ appearing as coefficients of the monomials. 
We note that the integral of each monomial inside the polynomial $\{\ba,\bb \}^n_1$
vanishes, since we have 
$$ \int_{\bC\bP^n} \omega_{FS}^{n-k}\wedge (d_B d^c_B u)^{k} = 0,  $$
for each $k=1,\cdots, n$ by the Stokes Theorem. 
Moreover, the part in a polynomial without the factor 
$ \{  \ba, \bb  \}_1^n $ will be called as \emph{the principal part} of this polynomial. 

\begin{lemma}
\label{lem-gr-0015}
For each integer $0\leq k\leq n$, 
there is a positive constant $B_k: = B(k, n)$ such that we have 
\begin{equation}
\label{gr-0067}
\int_{\bC\bP^n} 
(\dot{u}_t)^{n+1-k} \omega_{FS}^{n-k}\wedge \Theta_2^{k}
\leq 
B_k M_A^n(u)  \int_{\bC\bP^n} 
(\dot{u}_t)\omega_{FS}^{n},
\end{equation}
for all $r = e^t < e^{-A}$. 
Moreover, the constants can be constructed inductively by setting $B_0 =1$, and satisfying 
$$ B_{k+1}= \sum_{j=0}^{k}  \left( \begin{array}{c} k \\ j \end{array} \right) B_j, $$
for each $k=0, \cdots, n-1$.
\end{lemma}
\begin{proof}
First we claim that the following estimate holds 
\begin{equation}
\label{gr-0068}
 \bc \ba^{n-k} \be^{k} \leq B_k M^k \bc \ba^n + \{\ba,\bb \}^n_1,  
\end{equation}
for each $k=0,\cdots, n$. Then we have 
\begin{equation}
\label{gr-0069}
 \bc^{n-k+1} \ba^{n-k} \be^{k} \leq B_k M^n \bc \ba^n + \{\ba,\bb \}^n_1,  
\end{equation}
for each $k$, and our result follows from the Stokes Theorem. 

In order to prove the claim,
we will invoke an induction argument on equation (\ref{gr-0068}).
First check for $k=0$, 
and hence we have  $B_0 =1$. 
Then we have for each $k=1, \cdots, n$
\begin{eqnarray}
\label{gr-0071}
\ba^{n-k} \be^k \bc &\leq& M \ba^{n-k} (\bb+\ba\bc)(\bb + M\ba)^{k-1}
\nonumber\\
&=& \{\ba,\bb \}^n_1 + M \ba^{n-k+1}\bc (\bb+ M\ba)^{k-1}
\nonumber\\
&\leq& \{\ba,\bb \}^n_1 + M \ba^{n-k+1}\bc (\be+ M\ba)^{k-1},
\end{eqnarray}
and then it follows from our induction hypothesis 
\begin{eqnarray}
\label{gr-0072}
&& M \ba^{n-k+1}\bc (\be+ M\ba)^{k-1} =
\sum_{j=0}^{k-1}  \left( \begin{array}{c} k-1 \\ j \end{array} \right)
M^{k-j} \ba^{n-j} \be^j \bc
\nonumber\\
&\leq&  M^k \sum_{j=0}^{k-1}  \left( \begin{array}{c} k-1 \\ j \end{array} \right) B_j \ba^n \bc + \{ \ba, \bb \}^n_1.
\end{eqnarray}
Here the constant $B_k$ is defined inductively by 
$$ B_{k}= \sum_{j=0}^{k-1}  \left( \begin{array}{c} k-1 \\ j \end{array} \right) B_j.$$
Combine with equation (\ref{gr-0071}) and (\ref{gr-0072}),
and then our claim follows. 

\end{proof}

Then we are ready to find a useful upper bound of the 
complex Monge-Amp\`ere mass via the decomposition formula 
and estimate in Lemma \ref{lem-gr-0015}.

\begin{prop}
\label{prop-gr-002}
For a function $u\in \cF^{\infty}(B_1)$, 
there exists a dimensional constant $C_n  \geq (n+1)$ satisfying 
\begin{equation}
\label{gr-007}
\frac{1}{\pi} \int_{B_r} (dd^c u)^{n+1} \leq C_n  M^n_A(u)\int_{\bC\bP^n} (\dot{u}_t) \omega^n_{FS},
\end{equation}
for each $r = e^t < e^{-A}$.
\end{prop}
\begin{proof}
Thanks to Theorem \ref{thm-df-001} and Corollary \ref{cor-gr-001}, 
 it is equivalent to estimate 
the following polynomial from the above  
\begin{equation}
\label{gr-008}
 \sum_{ j=0}^n  \left( \begin{array}{c} n+1 \\  j+1   \end{array} \right) 
(-1)^j  \be^{n-j} \bc^{j+1} \ba^j, 
\end{equation}
and we note that only the principle part will matter in the end.

\textbf{Step 1}. 
Drop off all the negative terms. 
More precisely, there are two cases for even and odd dimensions.
For instance, the polynomial in equation (\ref{gr-008}) 
can be estimated from the above, if $n = 2m$ is even 
\begin{equation}
\label{gr-010}
\sum_{ k=0}^m  \left( \begin{array}{c} 2m+1 \\  2k+1   \end{array} \right) 
 \be^{2m - 2k} \bc^{ 2k+1} \ba^{2k}.
\end{equation} 
On the other hand, if $n= 2m +1$ is odd, we have  
\begin{equation}
\label{gr-011}
\sum_{ k=0}^m  \left( \begin{array}{c} 2m+2 \\  2k+1   \end{array} \right) 
 \be^{2m - 2k +1} \bc^{ 2k+1} \ba^{2k}.
\end{equation} 

Then everything boils down to estimate the upper bounds 
of the polynomials in equation (\ref{gr-010}) and (\ref{gr-011}). 
In the following, we will deal with the even case first. 

\textbf{Step 2}.
For an even $n=2m$,
we apply Lemma \ref{lem-gr-0015} to equation (\ref{gr-010}) as 
\begin{eqnarray}
\label{gr-012}
&& \sum_{ k=0}^m  \left( \begin{array}{c} 2m+1 \\  2k+1   \end{array} \right) 
 \bc^{ 2k+1} \ba^{2k}  \be^{2m - 2k}
 \nonumber\\
 &\leq& 
 M^n  \sum_{ k=0}^m  \left( \begin{array}{c} 2m+1 \\  2k+1   \end{array} \right) B_{2m-2k} \bc \ba^{2m}
 + \{ \ba, \bb \}^{2m}_1.
\end{eqnarray} 
Thanks to the Stokes Theorem, 
our result follows with the dimensional constant  
$$ C_{2m}:=  \sum_{ k=0}^m  \left( \begin{array}{c} 2m+1 \\  2k+1   \end{array} \right) B_{2m-2k} .$$

\textbf{Step 3}.
For an odd $n=2m+1$,
we apply Lemma \ref{lem-gr-0015} to equation (\ref{gr-011}) as 
\begin{eqnarray}
\label{gr-013}
&& \sum_{ k=0}^m  \left( \begin{array}{c} 2m+2 \\  2k+1   \end{array} \right) 
 \bc^{ 2k+1} \ba^{2k}  \be^{2m - 2k+1}
 \nonumber\\
 &\leq& 
 M^n  \sum_{ k=0}^m  \left( \begin{array}{c} 2m+2 \\  2k+1   \end{array} \right) B_{2m-2k+1} \bc \ba^{2m+1}
 + \{ \ba, \bb \}^{2m+1}_1.
\end{eqnarray} 
Then 
our result follows from the Stokes Theorem again, 
with the dimensional constant  
$$ C_{2m+1}:=  \sum_{ k=0}^m  \left( \begin{array}{c} 2m+2 \\  2k+1   \end{array} \right) B_{2m-2k+1} .$$

Finally, it is straightforward to check that the constant $C_{n}$ is no smaller than $(n+1)$ in both cases.

\end{proof}

The inequality in Proposition \ref{prop-gr-002} is not sharp,
since all the negative terms have been dropped. 
In the following examples, 
we calculate the dimensional constant $C_n$
for small $n$'s.  

\begin{example}
\label{ex-1}
For $n=1$ odd and $m=0$, we can first compute
\begin{equation}
\label{ex-001}
 \bc^2\ba + 2 \bc\bb = 2 \be \bc - \ba\bc^2 \leq 2\be\bc, 
\end{equation}
and then it follows
\begin{eqnarray}
\label{ex-002}
 2 \be \bc &\leq& 2M (\bb + \ba\bc)
 \nonumber\\
 &=&  2M\bb + 2M\ba\bc
 \nonumber\\
 &=& \{\ba, \bb\}^1_1 + 2M\ba\bc.
\end{eqnarray}
Then we have the dimensional constant $C_1 = 2$. 
Comparing with the upper bound that we obtained in the previous work \cite{Li23}, 
there is no surprise that the estimate has been improved, 
since a more accurate 
positivity condition (Lemma \ref{lem-ps-001}) has been used. 

Moreover, we note that the above estimate is not sharp, 
since we have removed all the negative terms. In fact, 
we can obtain a better estimate as 
\begin{equation}
\label{ex-0012}
 \bc^2\ba + 2 \bc\bb \leq   \{\ba, \bb\}^1_1 + 2M\ba\bc - \ba\bc^2.
\end{equation}

\end{example}

\begin{example}
\label{ex-2}
For $n=2$ even and $m=1$, we compute 
\begin{eqnarray}
\label{ex-003}
&&\bc^3 \ba^2 + 3 \bc^2 \ba\bb + 3 \bc\bb^2
\nonumber\\
&=& \bc^3 \ba^2 - 3 \bc^2 \ba \be +3 \bc \be^2 \leq  M^2\ba^2\bc + 3 \bc \be^2,
\end{eqnarray}
and then it follows 
\begin{eqnarray}
\label{ex-004}
3 \bc \be^2 &\leq& 3 M (\bb+ \ba\bc) (\bb + M\ba)
\nonumber\\
&\leq& \{\ba,\bb \}^2_1 + 3 M \ba\bc (\be + M\ba)
\nonumber\\
&=&  \{\ba,\bb \}^2_1 + 3 M \ba\bc \be + 3 M^2\ba^2 \bc.
\end{eqnarray}
Plug equation (\ref{ex-002}) into the last line on the R.H.S. of the above equation, 
and we obtain 
\begin{equation}
\label{ex-005}
 3 \bc\be^2 \leq \{\ba,\bb \}^2_1 +   6 M^2\ba^2 \bc. 
 \end{equation}
Then we have the dimensional constant $C_2 =7$.
\end{example}

\begin{example}
\label{ex-3}
For $n=3$ odd and $m=1$, we compute 
\begin{eqnarray}
\label{ex-006}
&&\bc^4 \ba^3 + 4 \bc^3 \ba^2\bb + 6 \bc^2\ba\bb^2 + 4\bc\bb^3
\nonumber\\
&\leq&  4 M^2\ba^2 \be \bc + 4 \be^3 \bc.
\end{eqnarray}
and then it follows 
\begin{eqnarray}
\label{ex-007}
\bc \be^3 &\leq& M (\bb+ \ba\bc) (\bb + M\ba)^2
\nonumber\\
&\leq& \{\ba,\bb \}^3_1 +  M \ba\bc (\be + M\ba)^2
\nonumber\\
&=&  \{\ba,\bb \}^3_1 +  M \ba\bc \be^2 +  2 M^2\ba^2 \be\bc + M^3 \ba^3 \bc.
\end{eqnarray}
Plug equation (\ref{ex-002}) and (\ref{ex-005}) into the last line on the R.H.S. of the above equation, 
and we obtain 
\begin{equation}
\label{ex-008}
  \bc\be^3 \leq \{\ba,\bb \}^3_1 +   5 M^3\ba^3 \bc. 
 \end{equation}
Then we have the dimensional constant $C_3 =24$.
\end{example}

\subsection{The results}
For a function $u\in \cF^{\infty}(B_1)$, 
recall that we have defined the non-negative functional 
$$ I(u_t): = \int_{\bC\bP^n} (\dot{u}_t)\omega_{FS}^n, $$
for all $t <0$.
Then we can rewrite the estimate 
in Proposition \ref{prop-gr-002}  as 
\begin{equation}
\label{gr-017}
\pi^{-1} \mbox{MA}(u)(B_r) \leq C_n M_A^n (u) \cdot I(u_t),
\end{equation}
for all $r = e^t < e^{-A}$.
As a direct consequence, 
our main result follows for an $S^1$-invariant plurisubharmonic function 
that is also $C^2$-continuous outside the origin. 

\begin{theorem}
\label{thm-gr-001}
For a function $u\in \cF^{\infty}(B_1)$, 
there is a dimensional constant $C_n \geq (n+1)$ satisfying 
\begin{equation}
\label{gr-015}
\tau_u(0)\leq  C_n [\lambda_u(0)]^n \cdot \nu_u(0). 
\end{equation}
\end{theorem}
\begin{proof}
First take $r\rightarrow 0^+$, i.e. $t\rightarrow -\infty$, 
 on both sides of equation (\ref{gr-007}) 
and we obtain 
\begin{equation}
\label{gr-016}
\tau_u(0) \leq  C_n M^n_A(u) \nu_u(0).
\end{equation}
Then the result follows by taking $A\rightarrow +\infty$. 

\end{proof}

As a corollary, we confirm the zero mass conjecture in this special case.

\begin{cor}
\label{cor-gr-002}
For a function $u\in \cF^{\infty}(B_1)$, 
we have 
$$
\nu_u(0) = 0   \Rightarrow       \tau_u(0) = 0.
$$
\end{cor}

\bigskip

For a general function $u\in \cF(B_1)$, we can use the approximation of the functional 
by its standard regularization 
$u_{\ep}: = u*\rho_{\ep}$.
Then we are ready to prove the main theorem. 

\begin{theorem}
\label{thm-gr-002}
For a function $u\in \cF(B_1)$, 
there is a dimensional constant $C_n \geq (n+1)$ satisfying 
\begin{equation}
\label{gr-017}
\tau_u(0)\leq  2C_n [\lambda_u(0)]^n \cdot \nu_u(0). 
\end{equation}
\end{theorem}
\begin{proof}
It is enough to prove the following estimate 
\begin{equation}
\label{gr-018}
\tau_u(0)\leq  C_n \left( 2 M_A(u) + \k  \right)^n \cdot \nu_u(0), 
\end{equation}
for each $A>1$ large and $\k >0$ small. 
Take two large constants as
$$2< A < 2A< B.$$
Thanks to Lemma \ref{lem-pfc-001} and 
Proposition \ref{prop-pfc-002},
we can extract 
a subsequence $u_{\ep_k}$ from the regularization of $u$ 
satisfying 
\begin{equation}
\label{gr-019}
\mbox{MA}(u_{\ep_k})(B_r) \rightarrow \mbox{MA}(u)(B_r)
\end{equation}
and 
\begin{equation}
\label{gr-020}
I(u_{\ep_k, t}) \rightarrow I(u_t),
\end{equation}
as $k\rightarrow +\infty$ for almost all $t = \log r\in [-B, -A]$. 
Meanwhile, Lemma \ref{lem-pfc-005} implies that we have 
a uniform control of the maximal directional Lelong numbers of the regularization as 
\begin{equation}
\label{gr-021}
M_{A'}(u_{\ep}) \leq 2 M_A(u) + C\ep,   
\end{equation}
for all $ 2A< A' <B $ and $\ep< e^{-2A}/2$.
Apply Proposition \ref{prop-gr-002} to the subsequence, 
and then we obtain the estimate
\begin{eqnarray}
\label{gr-022}
\pi^{-1} \mbox{MA}(u_{\ep_k})(B_r) &\leq& C_n M_{A'}^n (u_{\ep_k}) I(u_{\ep_k, t})
\nonumber\\
&\leq& C_n  \left( 2M_A(u) + Ce^{-2A}/2 \right)^n  I(u_{\ep_k, t}),
\end{eqnarray}
for all $-2A < t< -B$ and $k$ large enough. 
Take $k\rightarrow +\infty$
on both sides of equation (\ref{gr-022}), 
and combine with equation (\ref{gr-019}) and (\ref{gr-020}). 
Hence we have the following inequality 
\begin{equation}
\label{gr-023}
\pi^n \tau_u(0) \leq  C_n  \left( 2M_A(u) + Ce^{-2A}/2 \right)^n  I(u_{t}),
\end{equation}
for almost all $t = \log r \in [-B, -2A]$.
Fix any $A > \log (C/ \k)$, and then 
equation (\ref{gr-018}) follows by taking $B\rightarrow +\infty$. 
Hence our main result follows.

\end{proof}

Finally, the zero mass conjecture directly follows from 
Proposition \ref{prop-pfc-001} and the above estimate
for an $S^1$-invariant plurisubharmoinc function.

\begin{theorem}
\label{cor-gr-003}
For a function $u\in \cF(B_1)$, 
we have 
$$
\nu_u(0) = 0   \Rightarrow       \tau_u(0) = 0.
$$
\end{theorem}

We remark that the estimate in Theorem \ref{thm-gr-002} can be improved.  
With a stronger apriori estimate on the maximal directional Lelong numbers (Remark 6.7, \cite{Li23}),
we can further obtain 
\begin{equation}
\label{gr-024}
\tau_u(0)\leq  C_n [\lambda_u(0)]^n \cdot \nu_u(0),
\end{equation}
for any function $u\in \cF(B_1)$. 
In $\bC^2$ with constant $C_1=2$,
one can directly check that the inequality in 
equation (\ref{gr-024}) is satisfied for all 
$S^1$-invariant examples provided in Section 6, \cite{Li23}.

However, this estimate is still not sharp, 
since all negative terms have been removed during the estimates. 
Finally, we note that 
this control of the residual Monge-Amp\`ere mass 
(by its Lelong number and maximal directional Lelong number) fails if the function is no longer $S^1$-invariant, see Example 6.13, \cite{Li23} and \cite{CG09}.

\subsection{Energy functionals}
\label{sub-007}
Through a variational approach, 
we have another point of view to describe 
the decomposition formula (Theorem \ref{thm-df-001})
and the result of the zero mass conjecture.

In fact,
we can think this formula  to be 
the push forward 
of the complex Monge-Amp\`ere measure 
of a $u\in\cF^{\infty}(B_1)$
from the K\"ahler cone $S^{2n+1}\times \bR_+$
to its base manifold $\bC\bP^n$.
To see this, re-write equation (\ref{df-016}) as follows:

\begin{equation}
\label{ef-001}
\pi^{-1}\mbox{MA}(u)(B_r) =  \sum_{k=0}^n
 \begin{pmatrix} n+1 \\ k\end{pmatrix}  E_{n,k}(u_t),
\end{equation}
where we define the functionals as 
$$ E_{n,k}(u_t): = \int_{\bC\bP^n} (\dot{u}_t)^{n+1-k} \omega_{FS}^{n-k} \wedge (i\ddbar_{\z} u_t )^k, $$
where $\d_{\z}$, $\dbar_\z$ and $\ddbar_\z$ mean 
operators on $\bC\bP^n$. 
In particular, 
the last term $E_{n,n}$
is related to the \emph{pluri-complex energy} $\cE$ (see \cite{BB22}, \cite{BB}) in the following way:
$$ (n+1) E_{n,n} (u_t) = -\frac{d}{dt}\cE(u_t), \ \ \ \text{where} \ \ \ \cE(u): = \int_{\bC\bP^n} (-u) (i\ddbar_{\z} u)^n.$$


In the domain case,
it is a well-known fact that
the second variation of $-\cE$ is 
the push forward of the complex Monge-Amp\`ere operator. 
Hence the energy $\cE$ is a concave functional along the so called \emph{sub-geodesic ray}. 
For more details about geodesics and sub-geodesics in the space of K\"ahler potentials, 
see \cite{Don97}, \cite{Sem74}, \cite{C00} and \cite{Li22}.

In our case, it is plausible to take $u_t$ for a function $u\in \cF^{\infty}(B_1)$ as a sub-geodesic ray 
with $C^2$-regularity, 
and it is a \emph{geodesic ray} if the complex Monge-Amp\`ere mass of $u$ vanishes, i.e. we have 
$$ (dd^c u)^{n+1}=0.  $$
Then we can take the primitives of the above functionals, 
and define the following energies (up to a constant) along a sub-geodesic ray: 
$$ \frac{d}{dt} \cE_{n,k}(u_t):= - \begin{pmatrix} n+1 \\ k \end{pmatrix} E_{n,k}(u_t). $$
In particular, we take $\cE_{n,n} = \cE$. 
Then equation \eqref{ef-001} can be rewritten as 
\begin{equation}
\label{ef-0015}
\cM = -\sum_{k=0}^n \cE_{n,k}, 
\end{equation}
along a sub-geodesic ray up to a constant, where $\cM$ is a $t$-primitive of the measure $\pi^{-1}\mbox{MA}(u)(B_r)$.
It is clear now that the decomposition formula 
induces a push forward of the complex Monge-Amp\`ere operator, under the Sasakian structure of the sphere.
That is to say,
 the non-triviality of the K\"ahler cone structure
is reflected by the terms like 
$\cE_{n, k}$ for $k=0,\cdots, n-1$ in the formula. 
Furthermore, 
it implies the following concavity of the energy. 

\begin{cor}
\label{cor-ef-001}
Along a sub-geodesic ray $u_t$ with $C^2$-regularity, 
the energy functional 
$ \sum_{k=0}^n \cE_{n, k} $
is concave. Moreover, 
it is affine if and only if $u_t$ is a geodesic ray. 
\end{cor}

The two dimensional version of the above Corollary has been 
proved in Section 7, \cite{Li23}, 
and this is its higher dimensional analogue.

Finally, we would like to point out that the first term $E_{n,0}$ in equation (\ref{ef-001}) 
also has a meaning. 
It is in fact a kind of $L^{p}$-Lelong number for $p=n+1$,
since we have the following convergence due to Lemma \ref{lem-pfc-003}: 
\begin{equation}
\label{ef-002}
\lim_{t\rightarrow -\infty} E_{n,0} (u_t) =  [\nu_u(0)]^{n+1}. 
\end{equation}
Recall that we introduce a primitive of the functional $I$ as 
$$ \cI(u_t): = \int_{\bC\bP^n} u_t \omega_{FS}^n. $$
Thanks to equation \eqref{ef-0015} and (\ref{ef-002}), we can rephrase the zero mass conjecture 
in terms of the energies. 

\begin{cor}
\label{cor-ef-002}
Along a sub-geodesic ray $u_t$ with $C^2$-regularity, 
assume that 
the asymptote of
the functional $\cI$ at $-\infty$ is zero.  
Then we have 
\begin{equation}
\label{ef-003}
\lim_{t\rightarrow -\infty}\frac{d}{dt} \cM(u_t)
=\lim_{t\rightarrow -\infty}\sum_{k=1}^n \frac{d}{dt} \cE_{n,k}(u_t)  =0.
\end{equation}
\end{cor}

It will be interesting to see if Corollary \ref{cor-ef-002}
also holds for a function $u$ in the family $\cF(B_1)$, 
i.e. along a bounded sub-geodesic ray. 
However, 
there is no a priori reason that these energies $\cE_{n, k}$ are well defined in that case.

\bigskip
\bigskip

\section{The method of moving frames}
\smallskip

In this section,  we are going to provide 
an alternative proof of the decomposition formula (Theorem \ref{thm-df-001})
via the method of moving frames, see \cite{C87}, \cite{Chern}.
Moreover, the complex hessian equation of a function without symmetry will be presented. 

In the language of moving frames,
a real coordinate $(x^A, y^A)$ of a point $p\in \bR^{2n+2}$
should be interpreted as a vector 
$$ x^A  \frac{\d}{\d x^A} + y^A \frac{\d}{\d y^A} \in T_p \bR^{2n+2}, $$
Then a complex coordinate $(z^A)$ corresponds to a vector in $T_p^{1,0}(\bC^{n+1})$,
and its complex conjugate $(\bar{z}^A)$ is a vector in $T_p^{0,1}(\bC^{n+1})$.
In the following, we are going to use the Einstein summation convention. 

\subsection{The structure equations}
Let $\{e_0,\cdots, e_n \}$ be a unitary field of $\bC^{n+1}$, 
which means that we have for $A, B = 0,\cdots, n$
\begin{equation}
\label{mf-001}
(e_A\ , \ e_B) = \delta_{AB},
\end{equation}
where the hermitian inner product is taken on $\bC^{n+1}$.
Taking the exterior derivative of the position vector $z$, we can write 
\begin{equation}
\label{mf-002}
dz = \omega^A e_A; \ \ \ \  \omega^A = a^{A}_B\cdot dz^B,
\end{equation}
where $(\omega^A)$ is a vector-valued $(1,0)$-form,
and $(a^A_B)$ is a unitary matrix with coefficients as smooth functions in $\bC^{n+1}$.
Then the Euclidean metric of $\bC^{n+1}$ is given as 
$$ ds^2_e =  \omega^A \overline{\omega^A}. $$
Hence it follow from equation (\ref{mf-001}) and (\ref{mf-002}) as 
\begin{equation}
\label{mf-003}
de_A = \omega^B_A e_B; \ \ \ \  \omega^B_C = a^B_A \cdot d \left(\overline{a^C_A}\right),
\end{equation}
where $(\omega^A_B)$ is a matrix-valued $1$-form on $\bC^{n+1}$, 
and then we have 
\begin{equation}
\label{mf-004}
\omega^A_B + \overline{\omega^B_A} =0.
\end{equation}
Taking exterior derivatives on both sides of equation (\ref{mf-003}) and (\ref{mf-002}), 
we obtain the following structure equations 
\begin{align}
\label{mf-005}
 d\omega^A  + \omega^A_B \wedge \omega^B = 0; 
 \\
 \label{mf-006}
 d \omega^A_B + \omega^A_C \wedge \omega^C_B =0.
\end{align}
The above equations reflect the fact that the Euclidean space is flat. 
Moreover, we can decompose the differential of a smooth function $u$ as 
$$ du = \d u + \dbar u, $$
where 
$$ \d u = u_A  \omega^A ; \ \ \  \dbar u =  u_{\bar A}\overline{\omega^A}; \ \ \  u_{\bar A} = \overline{(u_A)}. $$

\subsection{The Fubini-Study metric}
Next we set the first element in the frame as the 
unit vector towards the point: 
$$ e_0:=  e^{-t} z, $$
where the variable is $t: =\log |z|$.
Take exterior derivatives, and we obtain 
\begin{equation}
\label{mf-007}
dz = e^t (dt + \omega^0_0 )\cdot e_0 + e^t \omega_0^{\a} \cdot e_{\a},
\end{equation}
for all $\a = 1, \cdots, n$. 
Then it follows 
\begin{equation}
\label{mf-008}
\omega^0 = e^t (dt + \omega^0_0); \ \ \ \ \  \omega^{\a} = e^t \omega^{\a}_0. 
\end{equation}

Fixing such an $e_0$, 
a local unitary frame field
$(e_0, e_1, \cdots, e_n)$ is in fact a local section 
of the principle bundle $U(n+1)$ over $\bC\bP^n$. 
That is to say, we have the projection map
$$ \pi: U(n+1) \rightarrow \bC\bP^n, $$
defined as
$$(e_0, e_1, \cdots, e_n ) \rightarrow [e_0], $$ 
where $[e_0]$ is the equivalent class of $z$ in $\bC\bP^n$. 
It is clear that each fiber of this projection map
is isomorphic to $U(1)\times U(n)$, 
and hence a local section of this principle bundle 
also gives such a frame field. 

Moreover, the Fubini-Study metric on $\bC\bP^n$ 
can be given by a direct computation as  
\begin{equation}
\label{mf-009}
ds^2_{FS} = \theta^{\a} \overline{\theta^{\a}}, \ \ \ \ \ \theta^{\a} = \omega_0^{\a}.
\end{equation}
Then it follows 
\begin{equation}
\label{mf-010}
\theta^{\a} = e^{-t} \omega^{\a},
\end{equation}
and the $(1,0)$-forms $\{ \theta^{\a}, 1\leq \a \leq n \}$ builds a local coframe field on $\bC\bP^n$. 
In fact, the $1$-form $\omega^0_0$ exactly corresponds to the contact $1$-form in Sasakian geometry, 
and the coframe field gives the decomposition of the complex structure 
under the K\"ahler cone structure of $(\bC^{n+1})^*$.

Take exterior derivatives on equation (\ref{mf-010}), 
and then we obtain the structure equation for this coframe field: 
\begin{eqnarray}
\label{mf-011}
d\theta^{\a} &=& \omega_0^0\wedge \omega_0^{\a} + \omega_0^{\b}\wedge\omega^{\a}_{\b}
\nonumber\\
&=& - (\omega^{\a}_{\b} - \omega_0^0\delta^{\a}_{\b})\wedge \omega_0^{\b}
\nonumber\\
&=& - \theta^{\a}_{\b} \wedge \theta^{\b},
\end{eqnarray}
where 
$ \theta^{\a}_{\b}: = \omega^{\a}_{\b} - \omega_0^0\delta^{\a}_{\b}$ satisfies 
$$ \theta^{\a}_{\b} + \overline{\theta^{\b}_\a} =0.  $$
This is the connection $1$-forms with respect to the coframe field $\{ \theta^{\a}, 1\leq \a \leq n\}$.
For a smooth function $v$ on $\bC\bP^n$, we can write as follows:
$$ \d v = v_{\a} \theta^{\a}; \ \ \ \dbar v = v_{\bar\a} \overline{\theta^{\a}}; \ \ \  \overline{(v_{\a})} = v_{\bar\a}. $$
Then we infer 
\begin{eqnarray}
\label{mf-012}
\ddbar v &=& d (v_{\bar\a} \overline{\theta^{\a}}) 
\nonumber\\
&=& \left( \ d v_{\bar\a} - v_{\bar\b} \overline{\theta^{\b}_{\a}} \ \right) \wedge \overline{\theta^{\a}}
\nonumber\\
&=& \nabla v_{\bar\a} \wedge \overline{\theta^{\a}},
\end{eqnarray}
where we put 
\begin{eqnarray}
\label{mf-013}
\nabla v_{\bar\a}:&=& d v_{\bar\a} - v_{\bar\b} \overline{\theta^{\b}_{\a}} 
\nonumber\\
 &=& v_{\gamma\bar\a} \theta^{\gamma} + v_{\bar{\gamma}\bar\a} \overline{\theta^{\gamma} }.
\end{eqnarray}
Thanks to the symmetry of $v_{\bar{\gamma}\bar\a} $, 
we obtain the complex hessian of $v$ as  
\begin{equation}
\label{mf-014}
i\ddbar v = i v_{ \gamma\bar\a}  \theta^{\gamma} \wedge \overline{\theta^{\a}}.
\end{equation}

\subsection{The complex hessian}
Take a unitary frame field $\{e_{A}, 0\leq A \leq n\}$
on $\bC^{n+1}$ satisfying 
$$e^t e_0 = z. $$
Then we can compute for a function $u\in \cF^{\infty}(B_1)$ as follows: 
\begin{eqnarray}
\label{mf-015}
\dbar u &=& u_{\bar{0}} \overline{\omega^0} + u_{\bar\a} \overline{\omega^{\a}}
\nonumber\\
&=& e^t \left\{   u_{\bar 0} (dt - \omega^0_0) + u_{\bar\a} \overline{\theta^{\a}}    \right\}.
\end{eqnarray}
Thanks to equation (\ref{mf-010}), 
we can decompose the exterior derivative on $\bC^{n+1}$ locally as 
\begin{align}
\label{mf-016}
du_{\bar 0} = u_{0, \bar 0}\omega^0 e^{-t} + u_{\bar 0, \bar 0} \overline{\omega^0} e^{-t} + d^T u_{\bar 0};
\\
du_{\bar \a} = u_{0, \bar\a}\omega^0 e^{-t} + u_{\bar 0, \bar\a} \overline{\omega^0} e^{-t} + d^T u_{\bar\a},
\end{align}
where $d^T$ denotes the exterior derivative in the transversal direction $\bC\bP^n$. 
Hence the complex hessian of $u$ can be calculated as 
\begin{eqnarray}
\label{mf-017}
\ddbar u &=& d (\dbar u)
\nonumber\\
&=& e^t dt \wedge \left( - u_{\bar 0} \omega^0_0 + u_{\bar\a} \overline{\theta^{\a}} \right)
\nonumber\\
&+& e^t \left\{   d u_{\bar 0} \wedge (dt - \omega^0_0) +  u_{\bar 0} \theta^{\b}\wedge \overline{\theta^{\b}}   \right\}
\nonumber\\
&+& e^t  d(u_{\bar\a} \overline{\theta^{\a}}).
\end{eqnarray}
Here we note that $u_{\bar 0} = u_{0}$ is real since the function $u$ is $S^1$-invariant, 
and then it follows on the hypersphere $S_r$ 
\begin{equation}
\label{mf-018}
du_{\bar 0}|_{S_r} = d^T u_{\bar 0};\  \  \ \ \ 
du_{\bar \a}|_{S_r} = d^T u_{\bar\a}.
\end{equation}
Hence we have the restriction 

\begin{equation}
\label{mf-019}
i \ddbar u|_{S_r} = i e^t \left\{  ( u_{\bar 0}\delta_{\a\b} + u_{\a\bar\b}) \theta^{\a}\wedge\overline{\theta^{\b}} 
+ d^T u_{\bar 0} \wedge \overline{\omega^0_0} \right\},
\end{equation}
and this is the same formula in our Lemma \ref{lem-df-001}. 
Moreover, it is clear to have 
\begin{equation}
\label{mf-020}
d^c u = i e^t \left \{  u_{\bar 0} \overline{\omega^0_0}   +  \frac{1}{2} \left( u_{\bar\a} \overline{\theta^{\a}} - u_{\a}\theta^{\a} \right)  \right\}
\end{equation}
Then it is the same way to compute the $(2n+1)$-form 
$$ d^c u \wedge (dd^c u|_{S_r})^n, $$
and the decomposition formula follows by
taking the integral on $\bC\bP^n \times \{ t\}$ and perform the integration by parts. 

\subsection{No symmetry}
When there is no symmetry of the function, the situation becomes more interesting. 
Then equation (\ref{mf-018}) fails,
and we need to directly compute from equation (\ref{mf-017}).
First we have 
\begin{align}
\label{ns-001}
dt = \frac{1}{2} e^{-t} (\omega^0 + \overline{\omega^0});
\\
\omega^0_0 = \frac{1}{2} e^{-t} (\omega^0 - \overline{\omega^0}),
\end{align}
and then it follows 
\begin{equation}
\label{ns-002}
dt\wedge \omega^0_0 = -\frac{1}{2} e^{-2t} \omega^0 \wedge \overline{\omega^0}.
\end{equation}

Hence the first term on the R.H.S. of equation (\ref{mf-017}) 
is equal to the following: 
\begin{equation}
\label{ns-003}
\frac{1}{2} e^{-t} u_{\bar 0} \omega^0 \wedge \overline{\omega^0} 
+ \frac{1}{2}  u_{\bar\a} \omega^0 \wedge \overline{\theta^{\a}} 
+ \frac{1}{2}  u_{\bar\a} \overline{\omega^0} \wedge \overline{\theta^{\a}}. 
\end{equation} 
The second term can be computed as 
\begin{eqnarray}
\label{ns-004}
&&du_{\bar 0} \wedge (dt - \omega^0_0)e^t 
\nonumber\\
&=& ( u_{ 0, \bar 0} \omega^0 e^{-t} + u_{ \bar 0, \bar 0} \overline{\omega^0} e^{-t} + d^T u_{\bar 0}) \wedge \overline{\omega^0}
\nonumber\\
&=& e^{-t} u_{0, \bar 0} \omega^0 \wedge \overline{\omega^0}  + u_{\a, \bar 0} \theta^{\a} \wedge \overline{\omega^0} 
+ u_{\bar\a, \bar 0} \overline{\theta^{\a}} \wedge \overline{\omega^0},
\end{eqnarray}
where we put 
\begin{equation}
\label{ns-005}
d^T u_{\bar 0}= u_{\a, \bar 0} \theta^{\a} 
+ u_{\bar\a, \bar 0} \overline{\theta^{\a}}.
\end{equation} 
Moreover, we can compute the third term as 
\begin{eqnarray}
\label{ns-006}
&&e^t d(u_{\bar\a} \overline{\theta^{\a}}) = e^t \left( du_{\bar\a} \wedge \overline{\theta^{\a}} + u_{\bar\a} d\overline{\theta^{\a}} \right)
\nonumber\\
&=& ( u_{0, \bar\a} \omega^0   + u_{\bar 0, \bar\a} \overline{\omega^0}  ) \wedge \overline{\theta^{\a}}
 +  e^t (d^T u_{\bar\a} + u_{\bar\b} \overline{\theta^{\b}_{\a}} ) \wedge \overline{\theta^{\a}}
 \nonumber\\
 &=&   u_{0, \bar\a} \omega^0 \wedge \overline{\theta^{\a}}  + u_{\bar 0, \bar\a} \overline{\omega^0} \wedge \overline{\theta^{\a}}
 + e^t u_{\gamma\bar\a} \theta^{\gamma} \wedge \overline{\theta^{\a}}.
\end{eqnarray}
Combining with equation (\ref{ns-003}), (\ref{ns-004}) and (\ref{ns-006}), 
we obtain the complex hessian as 
\begin{eqnarray}
\label{ns-007}
\ddbar u &=&  e^{-t} \left(  u_{0, \bar 0} + \frac{1}{2} u_{\bar 0}  \right) \omega^0 \wedge \overline{\omega^0}
\nonumber\\
&+& \left(  u_{0, \bar\a} + \frac{1}{2} u_{\bar\a}   \right)   \omega^0 \wedge  \overline{\theta^{\a}}
+  u_{\a, \bar 0} {\theta^{\a}}\wedge \overline{\omega^0}
\nonumber\\
&+&   \left(  u_{\bar\a, \bar 0}  - u_{\bar 0, \bar\a} - \frac{1}{2} u_{\bar\a} \right) \overline{\theta^{\a}}\wedge \overline{\omega^0}
\nonumber\\
&+&   e^t  ( u_{\bar 0}\delta_{\a\b} + u_{\a\bar\b}) \theta^{\a}\wedge\overline{\theta^{\b}}.
\end{eqnarray}
Since the complex hessian is a $(1,1)$-form, it follows 
\begin{equation}
\label{ns-008}
 u_{\bar\a, \bar 0}  = u_{\bar 0, \bar\a} + \frac{1}{2} u_{\bar\a}.
\end{equation} 
Moreover, we can use the identity $ - \ddbar u = \dbar\d u$ to obtain 
the following commutation relations: 
\begin{align}
\label{ns-009}
 u_{\bar 0} \delta_{\a\b} + u_{\a\bar\b} = u_0 \delta_{\a\b} + u_{\bar\b\a};
 \\
 u_{ 0, \bar 0} + \frac{1}{2} u_{\bar 0} =  u_{\bar 0, 0} + \frac{1}{2} u_{0};
 \\
  u_{\a, 0}  = u_{0, \a} + \frac{1}{2} u_{\a},
  \\
 u_{ \bar\a, 0} = u_{0, \bar\a} + \frac{1}{2} u_{\bar\a} ; \ \ \ \   u_{\a, \bar 0} = u_{\bar 0, \a} + \frac{1}{2} u_{\a}.
\end{align}

Finally, we end up with the following complex hessian equation for a general function.
\begin{theorem}
\label{thm-ns-001}
For a smooth function $u$ on $(\bC^{n+1})^*$, 
its complex hessian can be decomposed as 
\begin{eqnarray}
\label{ns-010}
\ddbar u &=&  e^{-t} \left(  u_{ 0, \bar 0} + \frac{1}{2} u_{\bar 0}  \right) \omega^0 \wedge \overline{\omega^0}
\nonumber\\
&+&  u_{\bar\a, 0}  \omega^0 \wedge  \overline{\theta^{\a}}
+  u_{ \a, \bar 0}   {\theta^{\a}}\wedge \overline{\omega^0} 
\nonumber\\
&+&   e^t  ( u_{\bar 0}\delta_{\a\b} + u_{\a\bar\b})   \theta^{\a}\wedge\overline{\theta^{\b}}.
\end{eqnarray}
In particular, we have
\begin{equation}
\label{ns-011}
 u_{ 0, \bar 0} + \frac{1}{2} u_{\bar 0} \geq 0 \ \ \ 
\text{and}\ \ \
( u_{\bar 0}\delta_{\a\b} + u_{\a\bar\b})  ( i\theta^{\a}\wedge\overline{\theta^{\b}} ) \geq 0,
\end{equation}
 if $u$ is further plurisubharmonic.
 
\end{theorem}

\end{document}